\documentclass[12pt]{article}
\usepackage{amsfonts,amsmath,amssymb}
\usepackage{slashed,setspace}
\usepackage[all]{xy}
\usepackage{fullpage}
\usepackage{enumerate}
\newtheorem{ithm}{Theorem}
\newtheorem{iprop}{Proposition}

\newtheorem{theorem}{Theorem}[section]
\newtheorem{definition}[theorem]{Definition}
\newtheorem{lemma}[theorem]{Lemma}
\newtheorem{remark}[theorem]{Remark}
\newtheorem{corollary}[theorem]{Corollary}
\newtheorem{proposition}[theorem]{Proposition}

\newtheorem{notation}[theorem]{Notation}

\def\leftnote#1{\vadjust{\setbox1=\vtop{\hsize 30mm\parindent=0pt\bf\baselineskip=9pt\rightskip=4mm plus 4mm#1}\hbox{\kern-3cm\smash{\box1}}}}

\newenvironment{proof}[1][Proof]{\par\addvspace{2mm}\noindent\textbf{#1.} }{\ \rule{0.5em}{0.5em}\par\vspace{4mm}}

\newenvironment{proof1}[1][Proof of Theorem~\ref{MainTheorem}]{\par\addvspace{2mm}\noindent\textbf{#1.} }{\ \rule{0.5em}{0.5em}\par\vspace{4mm}}

\newenvironment{proof2}[1][Proof of Proposition~\ref{LT_div_pt_thm}]{\par\addvspace{2mm}\noindent\textbf{#1.} }{\ \rule{0.5em}{0.5em}\par\vspace{4mm}}
\newenvironment{proof3}[1][Proof of Proposition~\ref{thm3}]{\par\addvspace{2mm}\noindent\textbf{#1.} }{\ \rule{0.5em}{0.5em}\par\vspace{4mm}}

\newenvironment{proof4}[1][Proof of Theorem~\ref{Thm4}]{\par\addvspace{2mm}\noindent\textbf{#1.} }{\ \rule{0.5em}{0.5em}\par\vspace{4mm}}

\newcommand{\Q}{\mathbb{Q}}
\newcommand{\Z}{\mathbb{Z}}
\newcommand{\F}{\mathbb{F}}

\newcommand{\A}{{\mathcal A}}

\newcommand{\W}{{\mathcal W}}
\newcommand{\bo}{{\mathcal O}}
\newcommand{\D}{{\mathfrak D}}

\newcommand{\bp}{{\mathcal P}}

\newcommand{\C}{\mathbb{C}}

\DeclareMathOperator{\Gal}{Gal}
\DeclareMathOperator{\End}{End}

\begin{document}
 \bibliographystyle{plain}  
 \title{Formal group exponentials and Galois modules in Lubin-Tate extensions}
 \author{Erik Jarl Pickett and Lara Thomas}
 \maketitle

\abstract{
Explicit descriptions of local integral Galois module generators in certain extensions of $p$-adic fields due to Pickett have recently been used to make progress with open questions on integral Galois module structure in wildly ramified extensions of number fields. In parallel, Pulita has generalised the theory of Dwork's power series to a set of power series with coefficients in Lubin-Tate extensions of $\Q_p$ to establish a structure theorem for rank one solvable p-adic differential equations.

In this paper we first generalise Pulita's power series using the theories of formal group exponentials and ramified Witt vectors. Using these results and Lubin-Tate theory, we then generalise Pickett's constructions in order to give an analytic representation of integral normal basis generators for the square root of the inverse different in all abelian totally, weakly and wildly ramified extensions of a p-adic field. Other applications are also exposed.}
\section*{Introduction}
The main motivation for this paper came from new progress in the theory of Galois module structure. Indeed, explicit descriptions of local integral Galois module generators due to Erez~\cite{erez} and Pickett~\cite{Pickett} have recently been used to make progress with open questions on integral Galois module structure in wildly ramified extensions of number fields (see \cite{PickettVinatier} and~\cite{Vinatier_jnumb}). 
\vskip1mm
Precisely, let~$p$ be a prime number. Pickett, generalising work of Erez, has constructed normal basis generators for the square root of the inverse different in degree~$p$ extensions of any unramified extension of $\Q_p$.  His constructions were obtained by using special values of Dwork's power series.  Moreover, they have recently been used by Pickett and Vinatier~\cite{PickettVinatier} to prove that the square root  of the inverse different of $E/F$ is free over $\Z[G]$ under certain conditions on both the decomposition groups of $G$ and the base field $F$, when $E/F$ is a finite odd degree Galois extension with group~$G$.
\vskip1mm
In parallel, Pulita has generalised the theory of Dwork's power series to a set of power series with coefficients in Lubin-Tate extensions of $\Q_p$ in order to classify rank one $p$-adic solvable differential equations~\cite{Pulita}.
\vskip1mm

Our main goal was to generalise Erez and Pickett's construction in order to give explicit descriptions of integral normal basis generators for the square root of the inverse different in all abelian totally, weakly and wildly ramified extensions of a $p$-adic field. In this paper, our goal is totally achieved using a combination of several tools~: formal group exponentials, Lubin-Tate theory, and the theory of ramified Witt vectors. This leads us to generalise Pulita's formal power series to power series with coefficients in Lubin-Tate extensions of any finite extension of $\Q_p$. 
\vskip1mm
At the same time, we also get explicit generators for the valuation ring over its associated order, in maximal abelian totally, weakly and wildly ramified extensions of any $p$-adic field.

\

\noindent
{\bf Notation.} Let $p$ be a rational prime, and let $\Q_p$ be the field of $p$-adic numbers, $\bar{\Q}_p$ be a fixed algebraic closure of $\Q_p$ and $\C_p$ be the completion of $\bar{\Q}_p$ with respect to the $p$-adic absolute value. We let $v_p$ and $|\ |_p$ be the normalised $p$-adic valuation and absolute value on $\mathbb{C}_p$ such that $v_p(p)=1$ and $|x|_p=p^{-v_p(x)}$. 
As $v_p$ and $|\ |_p$ are completely determined by each other, either can be used in the statement of results;
we will use the valuation $v_p$ as this is the convention in the literature on Galois modules in Lubin-Tate extensions.
\vskip1mm
Throughout this paper, for any extension $K/\Q_p$ considered we will always assume $K$ is contained in $\bar{\Q}_p$ and
  we will denote by $\bo_K$, $\bp_K$ and $k$ its valuation ring, maximal ideal and residue field 
  respectively. We identify the residue field of $\Q_p$ with the field of $p$~elements, $\F_p$. 
  For any $n\in\mathbb{Z}_{>0}$, we denote by $\mu_n$ the
group of $n$th roots of unity contained in~$\bar{\Q}_p^{\times}$.

\

\noindent
{\bf Presentation of the paper.}  Let $\gamma\in\bar{\Q}_p$ be a root of the polynomial $X^{p-1}+p$. Dwork's exponential power series with respect to~$\gamma$ is defined as
$$E_{\gamma}(X)=\exp(\gamma (X-X^p))\enspace \in 1+X \Z_p[[X]],$$
where the right hand side is the composition of the two power series $\gamma (X-X^p)$ and $\exp(X)$.
Dwork's power series is \textit{over-convergent}, in the sense that it converges with respect to $|\ |_p$ 
 on an open disc $\{x\in{\mathbb C}_p:v_p(x)>c \}$ for some $c<0$ (\cite{LangII}, Chap.~14, \S 2, Remark after Lem.~2.2); it also has the property that $E_{\gamma}(1)$ is equal to a primitive $p$th root of unity in $\bar{\Q}_p$~(\cite{LangII}, Chap.~14, \S 3, Thm~3.2).
 
\vskip1mm
Dwork's power series was recently generalised by Pulita~\cite{Pulita} to a set of power series
with coefficients in Lubin-Tate
extensions of $\Q_p$: Let $f(X)\in\Z_p[X]$ be some Lubin-Tate
polynomial with respect to the uniformising parameter $p$,
\textit{i.e.},
$$P(X)\equiv X^p\mod p\Z_p[X] \text{\ \ \ and\ \ \ } P(X)\equiv pX\mod
X^2\Z_p[X]\enspace.$$  
Let $\{\omega_i\}_{i>0}$ be a \textit{coherent set of roots}
associated to $P(X)$, namely a sequence of elements of $\bar\Q_p$ such
that $\omega_1\ne0$, $P(\omega_1)=0$ and $P(\omega_{i+1})=\omega_{i}$; we refer to $\omega_n$ as an $n$th Lubin-Tate division point with respect to $P$. For $n\in\Z_{>0}$, Pulita defines the exponentials 
$$E_{P,n}(X)=\exp\left(\sum_{i=0}^{n-1}\frac{\omega_{n-i}(X^{p^i}-X^{p^{i+1}})}{p^{i}}\right) \in 1+X\Z_p[\omega_{n}][[X]]\enspace.$$ 
\vskip1mm
This generalises Dwork's power series as $E_{P,1}(X)=E_{\gamma}(X)$ when $P(X)=X^p+pX$.
For all choices of $P(X)$ and $n$, the power series $E_n(X)$ is over-convergent and has the property that $E_n(1)$ is a primitive $p^n$th root of unity $\zeta_{p^n}$ in $\bar{\Q}_p$. Comparing degrees then shows us that $\Q_p(\omega_n)=\Q_p(\zeta_{p^n})$ for all $n$ and all choices of $P(X)$.
We remark that this result is also a consequence of basic Lubin-Tate theory, see \cite{Lubin_Tate} or \cite{serre-lubintate} for details of this theory.

\

In this paper, we first generalise Pulita's exponentials to power series with coefficients in Lubin-Tate extensions of any finite extension~$K$  of $\Q_p$, in particular by combining Fr\"ohlich's notion of a formal group exponential (\cite{Frohlich}, Chap.~IV, \S 1) with the theory of ramified Witt vectors. Note that we impose no other restrictions on our base field $K$ and no restrictions on the uniformising parameter used to construct the Lubin-Tate extensions of~$K$. Inspired by the methods of Pulita, we prove the following core result of the paper~:

 \begin{ithm}\label{MainTheorem}
Let $K$ be a finite extension of $\Q_p$,
with valuation ring, maximal ideal and residue field denoted by $\bo_K$, $\mathcal P_K$ and $k$ respectively. Let $q=\text{card}(k)$ for some power $q$ of $p$.

Let $\pi$ and $\pi'$ be two uniformising parameters for $\bo_K$, and let $P,Q \in \bo_K[[X]]$ be Lubin-Tate polynomials with respect to $\pi$ and $\pi'$ respectively, \textit{i.e.},
$$P(X)\equiv X^q\mod \pi\bo_K[X] \text{\ \ \ and\ \ \ } P(X)\equiv \pi X\mod
X^2\bo_K[X]\enspace.$$ 

We write $F_P \in \bo_K[[X,Y]]$ for the unique formal group that admits $P$ as an endomorphism and $\exp_{F_P}\in XK[[X]]$ for the unique power series such that
$$\exp_{F_P}(X+Y)=F_P(\exp_{F_P}(X),\exp_{F_P}(Y)).$$
Let $\{\omega_i\}_{i> 0}$ be a coherent set of roots associated to $Q$.
\begin{enumerate}
\item
For every $n\geq 1$, the formal power series 
$$\mathcal E_{P,n}^{Q}(X):=\exp_{F_P}\left(\sum_{i=0}^{n-1}\frac{\omega_{n-i}(X^{q^i}-X^{q^{i+1}})}{\pi^i}\right)$$
lies in $X\bo_K[\omega_n][[X]]$, and is over-convergent if $\pi\equiv\pi'\mod \mathcal P_{K}^{n+1}$.
\item Moreover, we have the congruence $\mathcal{E}_{P,n}^{Q}(X)\equiv \omega_nX\mod \omega_n^2X\bo_K[\omega_n][[X]]$.
\end{enumerate}
\end{ithm}


To compare to Pulita's result, we remark that if $K=\Q_p$, $\pi=\pi'=p$, and $P(X)=(X+1)^p-1$, then $\mathcal{E}^{Q}_{P,n}(X)=E_{Q,n}(X)-1$. 

\

We then apply Theorem~\ref{MainTheorem} to give two explicit results in Lubin-Tate theory. For each integer $n\geq1$, we denote by $K_{\pi,n}$ the $n$-th Lubin-Tate extension of $K$ with respect to $\pi$. This extension is abelian and totally ramified, with degree $q^{n-1}(q-1)$ and conductor~$n$. Let $P\in \bo_K[X]$ be a Lubin-Tate polynomial with respect to $\pi$. The extension $K_{\pi,n}/K$ is generated by any primitive $n$-th Lubin-Tate division point with respect to $P$,  \textit{i.e.}, any element $\omega \in \bar{\Q}_p$ such that $P^{(n)}(\omega)=0$ whereas $P^{(n-1)}(\omega)\not=0$.

\

As a first application of~Theorem~\ref{MainTheorem}, we give an analytic representation of Lubin-Tate division points as values of the power series $\mathcal{E}_{P,n}^Q(X)$ for all $n \geq 0$.
Precisely, keeping the same notation, we prove~:

\begin{iprop}\label{LT_div_pt_thm}
\begin{enumerate}
\item If  $\pi\equiv\pi'\mod\bp_K^{n+1}$, then $\mathcal{E}^{Q}_{P,m}(1)$ is a primitive $m$th Lubin-Tate division point with respect to $P$, for all integers $m$ with $0<m\le n$. 
\item If $\pi'=\pi$, then $\{\mathcal{E}^{Q}_{P,i}(1)\}_{i>0}$ is a coherent set of roots associated to $P(X)$.
\end{enumerate}

\end{iprop}

Another application of Theorem~\ref{MainTheorem} is concerned with an explicit description of the action of the Galois group $\text{Gal}(K_{\pi,n}/K)$ over the Lubin-Tate extension $K_{\pi,n}$ for all $n\geq 1$. Indeed, since $\mathcal{E}_{P,n}^{Q}(1)$ is a primitive $n$th Lubin-Tate division point with respect to $P$, from standard theory (see \cite{Iwasawa}, \S  6-7, specifically Theorem 7.1) we know that the elements of $\Gal(K_{\pi,n}/K)$ are those automorphisms such that 
$\mathcal{E}_{P,n}^{Q}(1)\mapsto [u]_P(\mathcal{E}_{P,n}^{Q}(1))$, where $u$ runs over a set of representatives of $\bo_K^{\times}/(1+\bp_K^n)$ --- here $[u]_P(X)$ is a specific power series in $X\bo_K[[X]]$, see Section \ref{fgroups} for full details.
Therefore, the following proposition gives a complete description of $\Gal(K_{\pi,n}/K)$ in terms of values of our power series.

 \begin{iprop}\label{thm3}
 Let  $\pi\equiv\pi'\mod\bp_K^{n+1}$.
 For $0\le i\le n-1$, let $z_i\in\mu_{q-1}\cup\{0\}$ with $z_0\ne0$.
Then, 
$$[\sum_{i=0}^{n-1}z_i\pi^i]_P(\mathcal E^{Q}_{P,n}(1))=\mathcal E^{Q}_{P,n}(z_0)+_{F_P}\mathcal E^{Q}_{P,n-1}(z_1)+_{F_P}\ldots+_{F_P}\mathcal E^{Q}_{P,1}(z_{n-1})\enspace.$$
\end{iprop} 

{Note that, when $z_1=\cdots=z_{n-1}=0$, this proposition implies the relation $$ \mathcal E^{Q}_{P,n}(z_0)= [z_0]_P(\mathcal E^{Q}_{P,n}(1))\enspace , \quad \text{for all } z_0 \in \mu_{q-1}. $$  This is a generalisation of (\cite{Lang}, Chap.~14, Thm.~3.2).}

\

Finally,  we shall prove our second main result, as a consequence of Theorem~\ref{MainTheorem} and Proposition~\ref{LT_div_pt_thm}~: we use the power series $\mathcal{E}_{P,2}^Q(X)$ to construct explicit normal basis generators in  abelian, weakly and wildly ramified extensions of any $p$-adic field.  In this manner, we generalise the constructions of Erez and Pickett, and give support towards the resolution of open questions in Galois module structure theory~:

\begin{ithm}\label{Thm4}
Let $K$ be a finite extension of~$\Q_p$, with valuation ring $\bo_K$, residue field $k$ and residue cardinality $q$.  Fix a uniformising parameter $\pi$ of $K$ and let $P(X)=X^q+\sum_{i=2}^{q-1}a_iX^i+\pi X$ be some Lubin-Tate polynomial of degree $q$ with respect to~$\pi$. Let  $M_{\pi,2}$ be a maximal abelian totally, weakly and wildly ramified extension of $K$. Let $\A_{M_{\pi,2}/K}$ denote the unique fractional ideal in $M_{\pi,2}$ whose square is equal to the inverse different of $M_{\pi,2}/K$ (see Section \ref{Galmod}).

Let $\pi'\in K$ be another uniformising element of $K$, with $\pi\equiv \pi' \mod \mathfrak p_K^3$~; and let $Q \in O_K[X]$ be a Lubin-Tate polynomial with respect to $\pi'$.
\vskip3mm
If $v_p(a_{q-1})=v_p(\pi)$, then
\begin{enumerate}
\item the trace element $Tr_{K_{\pi,2}/M_{\pi,2}}(\mathcal{E}^{Q}_{P,2}(1))$ is a uniformising parameter of $M_{\pi,2}$ and a generator of the valuation ring $\mathcal O_{M_{\pi,2}}$ of $M_{\pi,2}$ over its associated order in the extension $M_{\pi,2}/K$~; 
\item if $p$ is odd, then the elements $$\frac{Tr_{K_{\pi,2}/M_{\pi,2}}(\mathcal{E}^{Q}_{P,2}(1))}{\pi} \text{\ \ \ and\ \ \ }\frac{Tr_{K_{\pi,2}/M_{\pi,2}}(\mathcal{E}^{Q}_{P,2}(1))+q}{\pi}$$ 
are both generators of $\A_{M_{\pi,2}/K}$ over $\bo_K[\Gal(M_{\pi,2}/K)]$ .
\end{enumerate}
\end{ithm}

\

Furthermore, Part~2 of this theorem will enable us to give explicit integral normal basis generators for the square root of the inverse different in every abelian totally, weakly and wildly ramified extension of any $p$-adic field~(see Corollary~\ref{CorThm2}).

\

We also remark that in this second part, the first element seems the more natural, however the second is in fact the generalisation of Erez's basis generator for the square root of the inverse different.  
If these basis generators can be used in local calculations in a similar way to those of Erez and Pickett, it should be possible to solve the case of whether $\A_{E/F}$ is free over $\Z[G]$ whenever the decomposition groups at wild places are abelian and in particular whenever $E/F$ itself is abelian. We hope this will be possible in the future, however, so far these calculations have eluded us.

\

\noindent
{\bf Organisation of the paper.}  This paper is organised into three sections. In Section~\ref{background}, we give the background to the theory we need to prove our  results. Precisely, we first introduce Lubin-Tate formal groups in their original setting and also in terms of Hazewinkel's so called functional equation approach. We also introduce Fr\"ohlich's notion of formal group exponentials and logarithms. Following work of Ditters, Drinfeld, and Fontaine and Fargues, we finally introduce the theory of ramified Witt vectors needed to generalise Pulita's methods.
In Section~\ref{sec2}, we study the properties of the power series $\mathcal{E}^{Q}_{P,n}(X)$ and prove Theorem~\ref{MainTheorem}. We also improve on Fr\"ohlich's original bound of the radius of convergence of any formal group exponential coming from a Lubin-Tate formal group over a non-trivial extension of~$\Q_p$. 
In Section~\ref{applications}, we explore the applications described above and prove Propositions~\ref{LT_div_pt_thm} and \ref{thm3}, as well as Theorem~\ref{Thm4}.



\section{Background} \label{background}
\subsection{Formal groups}\label{fgroups}

Let $A$ be a commutative ring with identity. 
\begin{definition}
We define a $1$-dimensional commutative formal group over $A$ to be a formal power series $F(X,Y)\in A[[X,Y]]$ such that
\begin{enumerate}
\item $F(X,Y)=F(Y,X)$
\item $F(X,F(Y,Z))=F(F(X,Y),Z)$
\item $F(X,0)=X=F(0,X)$
\end{enumerate}
Throughout, all the formal groups we consider will be $1$-dimensional commutative formal groups. For brevity we will now refer to these simply as formal groups.
\end{definition}

Properties 1-3 can be used to prove that there exists a unique $j(X)\in XA[[X]]$ such that $F(X,j(X))=0$ (see Appendix A.4.7 of \cite{Hazewinkel_book}). 
This means that the formal group $F(X,Y)$ endows $XA[[X]]$, among other sets, with an abelian group structure.

\begin{notation}
When considering a set endowed with such a group structure we write the group operation as $+_F$. We will also use the notation $-_F$ to determine this group operation composed with the group inverse, for example $$F(X,Y)=X+_FY\textit{\ \ \ and\ \ \ }F(X,j(Y))=X-_FY\enspace.$$
\end{notation}

\begin{definition}
Let $F(X,Y)$ and $G(X,Y)$ be two formal groups over $A$. A homomorphism over the ring $A$, $f:F(X,Y)\rightarrow G(X,Y)$, is a formal power series $f(X)\in X A[[X]]$ such that 
$$f(F(X,Y))=G(f(X),f(Y))\enspace.$$
Moreover, we say that the homomorphism $f$ is an isomorphism if there exists a homomorphism $f^{-1}:G(X,Y)\rightarrow F(X,Y)$ such that $f(f^{-1}(X))=f^{-1}(f(X))=X$.
\end{definition}

\subsubsection*{Lubin-Tate formal groups}

We now describe a special type of formal group, due originally to Lubin and Tate. Such formal groups are used in local class field theory to construct maximal totally ramified abelian extensions of a $p$-adic field $K$. For full details see, for example, \cite{Lubin_Tate} or \cite{serre-lubintate}.

\

Let $K$ be a finite extension of $\Q_p$, fix a uniformising parameter $\pi$ of $\bo_K$ and let $q=|\bo_K/\bp_K|$ be the cardinality
of the residue field of $K$. 

\begin{definition}
We define $\mathcal{F}_{\pi}$ as the set of formal power series $P(X)$ over $\bo_K$ such that
$$P(X)\equiv\pi X\mod X^2\bo_K[[X]]\text{\ \ \ and\ \ \ } P(X)\equiv X^q\mod\pi\bo_K[[X]]\enspace.$$

 Such power series are called Lubin-Tate series with respect to~$\pi$. 
\end{definition}

For each $P\in\mathcal{F}_\pi$ there exists a unique formal group $F_P(X,Y)\in\bo_K[[X,Y]]$ 
which admits $P$ as an endomorphism. Such formal groups are known as Lubin-Tate formal groups.
For each $P\in\mathcal{F}_{\pi}$ and each $a\in \bo_K$,
there exists a unique formal power series, $[a]_P(X)\in
X\bo_K[[X]]$, such that $P([a]_P(X))=[a]_P(P(X))$ and 
$$[a]_P(X)\equiv aX\mod 
X^2\bo_K[[X]]\enspace.$$

Further, the map $a\mapsto[a]_P(X)$ is an injective ring homomorphism $\bo_K\rightarrow \End_{\bo_K}(F_P)$ and for any $P,Q\in\mathcal{F}_\pi$, the formal groups $F_P(X,Y)$ and $F_Q(X,Y)$ are isomorphic over $\bo_K$.

\

Let $\bp_{\mathbb{C}_p}=\{x\in\mathbb{C}_p:v_p(x)>0\}$. For $P(X)\in\mathcal{F}_\pi$ and $a\in\bo_K$, the formal power series $F_P(X,Y)$ and $[a]_P(X)$ converge to limits in $\bp_{\mathbb C_p}$ when evaluated at elements of $\bp_{\mathbb C_p}$. We can thus use the abelian group operation $+_F$ and the injective ring homomorphism $a\mapsto[a]_P(X)$ to endow $\bp_{\mathbb C_p} $ with an $\bo_K$-module structure. 
For every $n\geq 1$, we then let 
$$T_{P,n}=\{x\in\bp_{\mathbb{C}_p}:[\pi^n]_P(x)=0\} $$ be the set of $\pi^n$-torsion points of this module and refer to it as the set of the \textit{$n$th Lubin-Tate division points} with respect to $P$. If $x\in T_{P,m}$ if and only if $m\le n$, then we say $x$ is a primitive $n$th division point.

\

We let
$$K_{\pi,n}=K(T_{P,n})\text{\ \ \ and\ \ \ } K_{\pi}= \cup_nK_{\pi,n}\enspace.$$
 The set $T_{P,n}$ depends on
the choice of the polynomial $P(X)$ but the field $K_{\pi,n}$
depends only on the uniformising parameter $\pi$. The extensions
$K_{\pi,n}/K$ are totally ramified, abelian and of degree $q^{n-1}(q-1)$. We have $K^{ab}=K_{\pi}K^{un}$ and $K^{un}\cap K_{\pi}=K$, where $K^{ab}$ and $K^{un}$ are the maximal abelian and unramified extensions of $K$ respectively.

\subsubsection*{Hazewinkel's approach to Lubin-Tate formal groups}

We now describe a different approach to the construction of Lubin-Tate formal groups due to Hazewinkel. 
This approach enables us to use Hazewinkel's \textit{functional equation lemma} to prove the integrality of various power series relating to formal groups, which will be essential in the sequel. For full details see \cite{Hazewinkel_book} (Chap.~I, \S2).
\

Recall that $p$ is a rational prime, $K$ is a finite extension of $\Q_p$, $\pi$ is a fixed uniformising parameter of $\bo_K$ and $q$ is the cardinality of the residue field of $K$. 
For any series $g(X)\in X\bo_K[[X]]$ we construct a new power series $f_g(X)\in XK[[X]]$ by the recursion formula (or functional equation)~: 

$$f_g(X)=g(X)+\frac{f_g(X^{q})}{\pi}\enspace.$$ 
We denote by $f_g^{-1}(X)\in XK[[X]]$ the unique power series such that 
$$f_g(f_g^{-1}(X))=X=f_g^{-1}(f_g(X))\enspace.$$ 
Note that if $f_g(X)\in X\bo_K[[X]]$ and the coefficient of $X$ in $f_g(X)$ is invertible in $\bo_K$, then $f_g^{-1}(X)\in X\bo_K[[X]]$, see \cite[A.4.6]{Hazewinkel_book}.

\

We now state two parts of the functional equation lemma for this special setting. For the full statement in all generality and its proof, see \cite[\S2.2-2.4]{Hazewinkel_book}.

\begin{theorem}[{Hazewinkel, \cite[2.2]{Hazewinkel_book}}]\label{fel}
Let $g(X),h(X)\in X\bo_K[[X]]$ and suppose that the coefficient of $X$ in $g(X)$ is invertible in $\bo_K$. Then,
\begin{enumerate}
\item $f_g^{-1}(f_g(X)+f_g(Y))\in \bo_K[[X,Y]]$.
\item $f_g^{-1}(f_{h}(X))\in X\bo_K[[X]]$.
\end{enumerate}

\end{theorem}

It is routine to check that $f_g^{-1}(f_g(X)+f_g(Y))$ is a formal group and from part 1. of the previous theorem we know that it has coefficients in $\bo_K$. In fact, it is a Lubin-Tate formal group and every Lubin-Tate formal group can be constructed in this manner. This link is described in the following proposition.

\begin{proposition}[{Hazewinkel, \cite[8.3.6]{Hazewinkel_book}}]
\label{hazewinkel_prop}
Let $g(X)\in X\bo_K[[X]]$ with 
$$g(X)\equiv X\mod X^2\bo_K[[X]].$$
Then,
\begin{enumerate}
\item $f_g^{-1}(\pi f_g(X))\in\mathcal{F}_{\pi}$.
\item If we let $P(X)=f_g^{-1}(\pi f_g(X))$, then $F_P(X,Y)=f^{-1}_g(f_g(X)+f_g(Y))$.
\item These relations give a one to one correspondence between the Lubin-Tate formal groups obtained from power series $P(X)\in\mathcal{F}_{\pi}$ and power series $g(X)\in X\bo_K[[X]]$ with $g(X)\equiv X\mod X^2\bo_K[[X]]$.
\end{enumerate}
\end{proposition}

We also observe that for any $a\in\bo_K$, substituting $h(X)=ag(X)$ into part $2$ of Theorem~\ref{fel} then gives us $f_g^{-1}(af_g(X))\in\bo_K[[X]]$, and so if $P(X)=f_g^{-1}(\pi f_g(X))$, then
$$
[a]_P(X)=f_g^{-1}(af_g(X))\enspace.
$$

\

\subsubsection*{Formal group exponentials}

Hazewinkel's power series $f_g(X)$ and $f_g^{-1}(X)$ can be thought of as special formal group isomorphisms which were first studied by Fr\"ohlich in (\cite{Frohlich}, Chap.~IV, \S1).

\

Let $E$ be any field of characteristic $0$, let $F(X,Y)$ be a formal group over $E$ and let $\mathbb{G}_a(X,Y)=X+Y$ be the additive formal group. There exists a unique isomorphism $\log_F:F\rightarrow\mathbb{G}_a$ over $E$ such that $\log_F(X)\equiv X\mod X^2E[[X]]$, known as the formal group logarithm (\textit{loc. cit.}, Prop.~1). The inverse of $\log_F(X)$ is known as the formal group exponential and is denoted by $\exp_F(X)$; we note that necessarily we also have, $\exp_F(X)\equiv X\mod X^2E[[X]]$. 

Now let $F(X,Y)=F_P(X,Y)=f_g^{-1}(f_g(X)+f_g(Y))$ be a Lubin-Tate formal group for $K$ as in Prop. \ref{hazewinkel_prop}. 
We then have
\begin{equation}f_g(X)=\log_{F}(X) \text{\ \ \ and\ \ \ }f_g^{-1}(X)=\exp_{F}(X)\label{logexp}\end{equation}
 and these power series are uniquely determined by the following equivalent identities:
\begin{equation}\label{logexp_identitites}
\begin{array}{ccc}
F(X,Y)&=&\exp_{F}(\log_{F}(X)+\log_{F}(Y))\\
\log_{F}(F(X,Y))&=&\log_{F}(X)+\log_{F}(Y)\\
\exp_{F}(X+Y)&=&F(\exp_{F}(X),\exp_{F}(Y))
\end{array}
\end{equation}
We also observe that,
\begin{equation}\label{a_action}
[a]_P(X)=\exp_{F}(a\log_{F}(X))\enspace.
\end{equation}

\

\begin{remark}
The reason these power series are referred to as formal group exponentials and formal group logarithms is that if $K=\Q_p$, then $P(X)=(X+1)^p-1\in\mathcal{F}_p$ and $F_P(X,Y)=X+Y+XY=\mathbb{G}_m$, the multiplicative formal group. We then have $\exp_{F_P}(X)=\exp(X)-1$ and $\log_{F_P}(X)=\log(X-1)$ where $\log$ and $\exp$ are the standard logarithmic and exponential power series. 
\end{remark}

\subsection{Witt vectors}\label{WittSection}
This section is concerned with the notion of ramified Witt vectors, generalising the classical theory of Witt vectors introduced by Witt in his original paper~\cite{Witt36}. This notion was first developed independently by Ditters~\cite{Ditters75} and Drinfeld~\cite{Drinfeld76}, and then by Hazewinkel~\cite{HazewinkelWittVectors} from a formal group approach in a more general setting. The reader is also referred to Section~5.1 of the current preprint~\cite{FontaineFargues} of Fontaine and Fargues.

\subsubsection*{Standard Witt vectors} 
We first briefly recall the construction of ``standard" Witt vectors. Let $p$ be a prime number, and let $X_0,X_1,...$ be a sequence of indeterminates. The original Witt polynomials are defined by~:
$$\forall n \geq 0, \quad \W_n(X_0,...,X_n)=\sum_{i=0}^np^i X_i^{p^{n-i}} \in \Z[X_0,...X_n]\enspace.$$

\

The standard Witt vectors can be constructed as a functor $W:A \mapsto W(A)$ from the category of commutative rings to itself. Precisely, if $A$ is a commutative ring, we first define $W(A)$ as the set of infinite sequences $A^{\Z_{\geq 0}}$. The  elements of $W(A)$ are called Witt vectors, and to each Witt vector $x=(a_n)_n\in W(A)$, one can attach a sequence $\langle a^{(n)} \rangle_n \in A^{\Z_{\geq 0}}$ whose coordinates are called the \textit{ghost components} of $x$ and are defined by the Witt polynomials~: $a^{(n)}=\W_n(a_0,...,a_n)$,  for all $n \geq 0$.

\

The set $W(A)$ is then uniquely endowed with two laws of composition that satisfy the axioms of a commutative ring, in such a way that the \textit{ghost map} $\Gamma_A : (a_n)_n \in W(A) \mapsto \langle a^{(n)}\rangle_n \in A^{\Z_{\geq 0}}$ becomes a ring homomorphism.

\

Under this functor, any ring homomorphism $\varphi:A \rightarrow B$ is sent to the ring homomorphism $W(\varphi):W(A)\rightarrow W(B)$ which is defined componentwise, \textit{i.e.}, $W(\varphi)((a_n)_n)=(\varphi(a_n))_n$. See (\cite{Bourbaki83}, Chap.~IX) for more details.

\

\subsubsection*{Ramified Witt vectors} Let $p$ be a prime number. Let $K$ be a finite extension of $\Q_p$, with valuation ring $\mathcal{O}_K$ and residue field $k$. We fix a uniformising parameter~$\pi$ of $\mathcal{O}_K$, and write $k=\F_q$ with $q=p^{f}$. Ramified Witt vectors over $\mathcal{O}_K$ are constructed as a functor $W_{\mathcal{O}_K,\pi}:A \mapsto W_{\mathcal{O}_K,\pi}(A)$ from the category of $\mathcal{O}_K$-algebras to itself, starting with generalised Witt-like polynomials and then proceeding along the lines of the construction of the usual Witt vectors. For convenience, as well as to collect some useful properties of the ramified Witt vectors, we shall briefly describe this functor. 

\

In the case of ramified Witt vectors, the relevant polynomials are~:
$$\forall n \geq 0, \quad \W_{n,\mathcal{O}_K,\pi}(X_0,...,X_n)=\sum_{i=0}^n\pi^iX_i^{q^{n-i}} \in \mathcal{O}_K[X_0,...,X_n]\enspace.$$

Let $A$ be an $\mathcal{O}_K$-algebra. We first define $W_{\mathcal{O}_K,\pi}(A)$ as the set $A^{\Z_{\geq 0}}$ as the set of infinite sequences over $A$. We shall use the notation $(a_n)_n$ for elements in $W_{\mathcal{O}_K,\pi}(A)$, and $\langle a_n\rangle_n$ for elements in $A^{\Z_{\geq 0}}$.

\

If $(a_n)_n \in W_{\mathcal{O}_K,\pi}(A)$, we define its ghost components as $a^{(n)}=\W_{n,\mathcal{O}_K,\pi}(a_0,...,a_n)$ for all $n \geq 0$. The sequence $\langle a^{(0)},a^{(1)},...\rangle$ is called the ghost vector of $(a_n)_n$. This defines a map, that we shall denote by $\Gamma_{\pi,\mathcal{O}_K,A}$ or simply by $\Gamma_A$ when the setup is explicit, called the ghost map of $A$~:
$$\begin{array}{cccc}
\Gamma_A \quad : & W_{\mathcal{O}_K,\pi}(A)& \longrightarrow & A^{\Z_{\geq 0}} \\
&  (a_n)_n  & \mapsto & \langle a^{(n)}\rangle_n\enspace.
\end{array}$$

\vskip3mm

The following lemma is essential for what follows.

\begin{lemma}\label{WittKey} Let $A$ be an $\mathcal{O}_K$-algebra with no $\pi$-torsion. Then, the ghost map $\Gamma_{A}$ is injective. 
If, moreover, there exists an $\mathcal{O}_K$-algebra endomorphism $\sigma : A \rightarrow A$ such that $\sigma(a)\equiv a^q \mod \pi A$ for all $a \in A$, then the image of $\Gamma_A$ is the sub-algebra of the product algebra $A^{\Z_{\geq 0}}$ given by 
$$\{\langle u_n\rangle_{n\geq 0} \in A^{\Z_{\geq 0}} \: : \:  \sigma(u_n) \equiv u_{n+1} \mod \pi^{n+1}A\}\enspace.$$
\end{lemma}

\begin{proof} We proceed along the lines of~(\cite{Bourbaki83}, No~2, Sect.~1, Par.~1 \& 2), replacing multiplication by~$p$ by multiplication by~$\pi$, and replacing $p$ by $q$ in the exponents. The first assertion is a consequence of the equivalence
$$(\star) \quad \begin{array}{ccl}
\Gamma_A((a_n)_n)=\langle u_n\rangle_n & \Leftrightarrow & \left\{\begin{array}{l} u_0=a_0 \\ u_{n+1}=\W_{n,\mathcal{O}_K,\pi}(a_0^q,...,a_n^q)+\pi^{n+1}a_{n+1}\enspace. \end{array}\right.
\end{array}$$
Therefore, for every sequence $\langle u_n\rangle_n \in A^{\Z_{\geq 0}}$, there exists at most one element $(a_n)_n \in W_{\mathcal{O}_K,\pi}(A)$ such that $\Gamma_A((a_n)_n)=\langle u_n\rangle_n$.

\vskip3mm

The second assertion is a consequence of the following relation that can easily be proved in the same way as Lemma~1 of (\cite{Bourbaki83}, No~2, Sect.~1), since $q \in \pi \mathcal{O}_K$~:
$$\forall x,y \in A, \, \forall n \geq 0,  \, \forall m \geq 1 \quad : \quad x \equiv y \mod \pi^m A \quad \Rightarrow \quad x^{q^n}\equiv y^{q^n} \mod \pi^{m+n}A\enspace.$$
In particular, for $m=1$ and for any sequence $(a_n)_n \in W_{\mathcal{O}_K,\pi}(A)$, this implies that 
$$\begin{array}{lcl}
\sigma(\W_{n,\mathcal{O}_K,\pi}(a_0,...,a_n)) & = & \W_{n,\mathcal{O}_K,\pi}(\sigma(a_0),...,\sigma(a_n)) \\
&\equiv  & \W_{n,\mathcal{O}_K,\pi}(a_0^q,...,a_n^q)\mod \pi^{n+1}A \\
&\equiv  & \W_{n+1,\mathcal{O}_K,\pi}(a_0,...,a_{n+1})\mod \pi^{n+1}A
\end{array}$$ 
Therefore, according to $(\star)$, we can prove by iteration on $n\geq 0$ that a sequence $\langle u_n\rangle_n$ is in the image of $\Gamma_A$ if and only if $\sigma(u_n)\equiv u_{n+1} \mod \pi^{n+1}A$ for all $n$. 
\end{proof}

In particular, the $\mathcal{O}_K$-algebra $A=\mathcal{O}_K[(X_n)_n,(Y_n)_n]$, endowed with the $\bo_K$-endomorphism $\sigma$ given by $\sigma(X_n)=X_n^q$ and $\sigma(Y_n)=Y_n^q$, satisfies the above lemma and is such that $\Gamma_A$ is bijective because the relation $\sigma(a)\equiv a^q \mod \pi A$ is satisfied for all $a \in A$. Therefore, the map $\Gamma_A$ transfers the structure of an $\bo_K$-algebra to $W_{\bo_K}(\mathcal{O}_K[(X_n)_n,(Y_n)_n])$. Moreover, for all $n \geq 0$ and all $x \in \bo_K$, this defines polynomials $S_n$ and $P_n$ in $\mathcal{O}_K[X_0,...,X_n,Y_0,...,Y_n]$,  $I_n$ and $C_{x,n}$
 in $\mathcal O_K[X_0,...,X_n]$ and $F_n$ in $\mathcal{O}_K[X_0,...,X_{n+1}]$ 
such that
$$\begin{array}{lcl}
\Gamma_A(S_0,S_1,...) & = & \Gamma_A(X_0,X_1,...)+\Gamma(Y_0,Y_1,...)\enspace, \\ 
\Gamma_A(P_0,P_1,...) & = & \Gamma_A(X_0,X_1,...) \times \Gamma_A(Y_0,Y_1,...)\enspace, \\
\Gamma_A(C_{x,0},C_{x,1},...) & = & x \cdot \Gamma_A(X_0,X_1,...) \enspace, \\
\Gamma_A(F_0,F_1,...) & = &(X^{(1)},X^{(2)},...)\enspace.
\end{array}
$$

\

Now, for any arbitrary $\mathcal{O}_K$-algebra $A$, we endow the set $W_{\mathcal{O}_K,\pi}(A)$ with laws of composition given by 
$$\forall a_n, b_n \in A, \: \forall x \in \bo_K, \: \left\{ \begin{array}{lcl}
(a_n)_n+(b_n)_n & = & (S_n(a_0,...,a_n,b_0,...,b_n))_n \\
(a_n)_n \times (b_n)_n & = &(P_n(a_0,...,a_n,b_0,...,b_n))_n \\
x\cdot (a_n)_n & = &(C_{x,n}(a_0,...,a_n))_n
\end{array}\right.   \enspace.$$

 Moreover, if $\varphi : A' \rightarrow A$ is any homomorphism of $\mathcal{O}_K$-algebras, we define the map $W_{\mathcal{O}_K,\pi}(\varphi):W_{\mathcal{O}_K,\pi}(A')\rightarrow W_{\mathcal{O}_K,\pi}(A)$ component-wise, \textit{i.e.}, $W_{\bo_K,\pi}(\varphi)((a_n)_n)=(\varphi(a_n))_n$. This map commutes with the previous laws of composition.

\

Next, for a fixed $\mathcal{O}_K$-algebra $A$, we consider the $\mathcal{O}_K$-algebra $B=\mathcal{O}_K[(X_a)_{a \in A}]$ which satisfies the assumptions of Lemma~\ref{WittKey} with $\sigma(X_a)=X_a^q$. In particular, the ghost map $\Gamma_B$ induces a bijection between $W_{\mathcal{O}_K,\pi}(B)$ and some subalgebra of $B^{\Z_{\geq 0}}$ that respects the previous laws of composition, which gives $W_{\mathcal{O}_K,\pi}(B)$ the structure of an $\mathcal{O}_K$-algebra. Now, the surjective homomorphism  $\rho: X_a \in B \mapsto a \in A$  yields a surjective map $W_{\mathcal{O}_K,\pi}(\rho):W_{\mathcal{O}_K,\pi}(B)\rightarrow W_{\mathcal{O}_K,\pi}(A)$, which endows $W_{\mathcal{O}_K,\pi}(A)$ with the structure of an $\mathcal{O}_K$-algebra as well, thereby proving the following:

\begin{proposition}[\cite{FontaineFargues}, Lemme~5.1]  The set-valued functor $\mathcal F : \{\mathcal{O}_K\text{-algebras}\} \rightarrow \text{Sets}$ given by $A \mapsto A^{\Z_{\geq 0}}$ factors through a unique $\mathcal{O}_K$-algebra-valued functor 
$$W_{\mathcal{O}_K,\pi}:  \{\mathcal{O}_K\text{-algebras}\} \rightarrow  \{\mathcal{O}_K\text{-algebras}\}$$ 
such that, for any $\mathcal{O}_K$-algebra $A$, the ghost map $\Gamma_{A}:(a_i)_{i\geq 0} \in W_{\mathcal{O}_K,\pi}(A) \mapsto (W_{n,O_K,\pi}(a_0,...,a_n))_{n\geq 0}$ is a homomorphism of $\mathcal{O}_K$-algebras.

In particular,  $W_{\bo_K,\pi}(A)$ is a $\bo_K$-algebra with Witt vector $(0,0,\ldots)$ as the zero element, and Witt vector $(1,0,0,\ldots)$ as the identity element.
\end{proposition}

\

An important remark is that, if $\pi'$ is another uniformising parameter, there exists a unique isomorphism  of functors, $u_{\pi,\pi'}$, between $W_{\mathcal{O}_K,\pi}$ and $W_{\mathcal{O}_K,\pi'}$, that commutes with the ghost maps (see~Section~5.1 of~\cite{FontaineFargues}). In particular, this is the reason why elements of $W_{\mathcal{O}_K,\pi}(A)$ are simply called \textit{ramified $\bo_K$-Witt vectors}.

\

Let $A$ be an $\mathcal{O}_K$-algebra. There are three maps that play a crucial role in $W_{\mathcal{O}_K,\pi}(A)$. The first is the Teichm\"uller lift $[\_]$, which is multiplicative and given by~:
$$[\_] \quad : \quad a \in A \mapsto [a]=(a,0,0,...) \in W_{\mathcal{O}_K,\pi}(A)\enspace.$$

\

The second is the Frobenius map $F$, defined uniquely by the polynomials $F_n$ introduced above. Precisely, as a consequence of Lemma~\ref{WittKey}, one can prove that this is the unique endomorphism of the  $\mathcal{O}_K$-algebra $W_{\mathcal{O}_K,\pi}(A)$ that satisfies~:
$$\forall (a_n)_n \in W_{\mathcal{O}_K,\pi}(A), \quad \Gamma_A(F(a_0,a_1,...))=\langle a^{(1)},a^{(2)},...\rangle\enspace.$$

As noticed in~\cite{FontaineFargues}, these two maps do not depend on $\pi$, in the sense that they commute with the isomorphism $u_{\pi,\pi'}$ for any other uniformising element $\pi'$.

\

The last map is the Verschiebung map $V_{\pi}$ and it is additive~:
$$V_{\pi}:(a_n)_n \in W_{\mathcal{O}_K,\pi}(A)\mapsto (0,a_0,a_1,...)\in W_{\mathcal{O}_K,\pi}(A)\enspace.$$
Contrary to the others, this map depends on the choice of $\pi$. Precisely, $V_{\pi'}=\frac{\pi'}{\pi}V_{\pi}$. Note also the relation $\Gamma_A(V_{\pi}(a_0,a_1,...))=\langle0,\pi a^{(0)}, \pi a^{(1)},\ldots\rangle$.

\

These maps satisfy the following properties, most of which can be proved after being translated to ghost components using Lemma~\ref{WittKey}~:

\begin{proposition}[(\cite{HazewinkelWittVectors}, Thm.~6.17), (\cite{FontaineFargues}, \S 5.1)] \label{WWF}Let $A$ be an $\mathcal{O}_K$-algebra, we have~:
\begin{enumerate}[i.]
\item The composed map $F V_{\pi}$ is the multiplication by $\pi$ in $W_{\mathcal{O}_K,\pi}(A)$,
 whereas the composed map $V_{\pi}F$ is the multiplication by $(0,1,0,0,...)$. When $A$ has $\pi$-torsion, these two operations correspond to each other.

\item For every $a=(a_0,a_1,...) \in W_{\mathcal{O}_K,\pi}(A)$, we have $F(a)\equiv a^q \mod \pi W_{\bo_K,\pi}(A)$, where $a^q$ is the $q$-th power of $a$ in $W_{\bo_K,\pi}(A)$ and $\pi W_{\bo_K,\pi}(A)$ is the ideal generated by $\pi (1,0,0,...)$. Moreover, if $F(a)=(\alpha_0,\alpha_1,...)$, then $\alpha_n\equiv a_n^q \mod \pi A$ for all $n \geq 0$.
\item If $l/k$ is a finite extension, then $W_{\mathcal{O}_K,\pi}(l)$ is the ring of integers of the unique unramified extension $L/K$ with residue field extension $l/k$.
\end{enumerate}
\end{proposition}

\vskip3mm

\begin{remark} In the language of Hazewinkel, these ramified Witt vectors are ``untwisted". In his paper~\cite{HazewinkelWittVectors}, Hazewinkel describes the functor of even more general Witt vectors, called ``twisted ramified Witt vectors", from a formal group law approach based on the functional equation lemma. These Witt vectors are obtained from the ramified Witt vectors by twisting the Teichm\"uller lift with some Lubin-Tate formal group law. Such a twist is necessary for Hazewinkel to describe all ramified discrete valuation rings with not necessarily finite residue fields.  See also Section~5.1.2 of~\cite{FontaineFargues} for a twisted version of ramified Witt vectors.
\end{remark}

\vskip3mm

\subsubsection*{Link with standard Witt vectors.}
Let $A$ be an $\mathcal{O}_K$-algebra. When $K=\Q_p$ and $\pi=p$, the ramified Witt $\Z_p$-algebra $W_{\Z_p,p}(A)$ is, as a ring, the ring of standard Witt vectors $W(A)$. In particular, one can prove the following (see, for example, the end of Paragraph~5.1 in~\cite{FontaineFargues})~: 

\begin{proposition} Let $K_0$ denote the maximal unramified subextension of $K/\Q_p$. If $A$ is a perfect $\F_q$-algebra, there is a canonical isomorphism~:
$$ W_{\mathcal{O}_K,\pi}(A)   \quad \rightarrow \quad W(A) \otimes_{O_{K_0}} \mathcal{O}_K \enspace,$$
under which the Teichm\"uller lifts correspond to each other, and the Frobenius map in $W_{\mathcal{O}_K,\pi}(A)$ is sent to $F^{f}\otimes \text{Id}$, where $f$ is the residue index of $K/\Q_p$ and $F$ denotes the standard Frobenius map in $W(A)$. 
\end{proposition}

\section{The power series $\mathcal E_{P,n}^{Q}(X)$}\label{sec2}

In this section we prove Theorem~\ref{MainTheorem}, the core result of the paper.

\subsection{Specific Witt vectors}\label{SpecificWittVectors}

 Recall that $p$ is a prime number and $K$ is a finite extension of $\Q_p$, 
with valuation ring, valuation ideal and residue field denoted by $\bo_K$, $\mathcal P_K$ and $k$ respectively. We let $q=\text{card}(k)$ and fix a uniformising parameter $\pi$ of $\bo_K$. In this section, we follow \cite[\S2.1]{Pulita}, but in our more general setting, in order to construct some specific ramified $\mathcal{O}_K$-Witt vectors with useful properties.

\

We fix a formal power series $P \in \mathcal{O}_K[[X]]$ such that 
\begin{eqnarray}\label{PSpecificWittVectors} P(0)=0 \quad \text{ and } \quad P(X)\equiv X^q\mod \pi \mathcal{O}_K[[X]]\enspace.
\end{eqnarray}
\noindent
This series defines an endomorphism of $\mathcal{O}_K$-algebras~:
$$\begin{array}{cccc}
\sigma_P \quad : & \mathcal{O}_K[[X]] & \rightarrow  & \mathcal{O}_K[[X]] \\
 & h(X) & \mapsto & h(P(X))
 \end{array}$$
with the property $\sigma_P(h)\equiv h^q \mod \pi \mathcal{O}_K[[X]]$ for all $h \in \mathcal{O}_K[[X]]$.

\

The following lemma is a straightforward generalisation of the first statement of \cite[Ch.IX, \S1, Exercise 14]{Bourbaki83} to ramified $\mathcal{O}_K$-Witt vectors~:
\begin{lemma} There is a unique homomorphism of $\mathcal{O}_K$-algebras
$$\begin{array}{cccc}
S_P \quad : & \mathcal{O}_K[[X]]  & \longrightarrow & W_{\mathcal{O}_K,\pi}(\mathcal{O}_K[[X]]) \\
& h & \mapsto & S_P(h)
\end{array}$$
such that, for every $h \in \mathcal{O}_K[[X]]$, the ghost vector of $S_P(h)$ is given by 
$$\langle h(X), h(P(X)), h(P(P(X))),...\rangle \in \mathcal{O}_K[[X]]^{\Z_{\geq 0}}\enspace,$$ 
\textit{i.e.}, such that the $n$-th ghost component of $S_P(h)$ is $\sigma_P^n(h)$, for every $n\geq 0$.
\vskip3mm
\noindent
Moreover, the  homomorphism of $\mathcal{O}_K$-algebras $S_P$ is also characterised by~:
$$F(S_P(h))=S_P(h(P))\enspace.$$

\end{lemma}
\begin{proof}
According to Lemma~\ref{WittKey}, since the $\mathcal{O}_K$-algebra $A:=\mathcal{O}_K[[X]]$ has no $\pi$-torsion and is endowed with the map $\sigma_P$, its ghost map $\Gamma_A : W(A) \rightarrow A^{\Z_{\geq 0}}$ is injective, and each sequence $\langle h(X), h(P(X)), h(P(P(X))),...\rangle$ is clearly in the image of $\Gamma_A$. Therefore, for every formal power series $h \in A$, there is a unique Witt vector $S_P(h) \in W_{\mathcal{O}_K,\pi}(A)$ with ghost components $\langle h(X), h(P(X)), h(P(P(X))),...\rangle$, thereby proving the existence of the map $S_P$. Finally, we prove that this is a homomorphism of $\mathcal{O}_K$-algebras after translating the properties of such a homomorphism in terms of ghost components, according to Lemma~\ref{WittKey}.

\vskip1mm
\noindent
The last assertion is an easy consequence of the definition of the map $F$ by the ghost components. \end{proof}

Now, let $L$ be a finite extension of $K$. We denote by $\mathcal{O}_L$ the valuation ring of $L$, and by $l$ its residue field. Let $a\in\bo_L$ be such that $v_p(a)>0$. By functoriality, the specialisation $\mathcal{O}_K[[X]] \rightarrow \mathcal{O}_L$, given by $X \mapsto a$, provides a homomorphism $W_{O_K,\pi}(\mathcal{O}_K[[X]])\rightarrow W_{O_K,\pi}(\mathcal{O}_L)$. For every such element $a \in \mathcal{O}_L$ and every formal power series $h\in \mathcal{O}_K[[X]]$, we shall denote by $S_{P,a}(h)$ the specialised Witt vector of $W_{\mathcal{O}_K,\pi}(\mathcal{O}_L)$ which is the image of $h$ via the composed homomorphism
$$\mathcal{O}_K[[X]] \begin{array}{c} {S_P} \\ \longrightarrow \\ \, \end{array} W_{\mathcal{O}_K,\pi}(\mathcal{O}_K[[X]]) \begin{array}{c} {X \mapsto a} \\ \longrightarrow \\ \, \end{array} W_{\mathcal{O}_K,\pi}(\mathcal{O}_L)\enspace.$$
In particular, note that the ghost vector of $S_{P,a}(h)$ is $\langle h(a),h(P(a)),h(P(P(a))),...\rangle$.

\

The following proposition is a key ingredient for what follows~:
\begin{proposition}\label{KeyProp} Let $h(X)=\sum_{i \geq 0}a_iX^i$ be a formal power series in $\mathcal{O}_K[[X]]$, let $a \in \mathcal{O}_L$ with $v_p(a) >0$ and write $S_{P,a}(h)=(\alpha_0,\alpha_1,...)\in W_{\mathcal{O}_K,\pi}(\mathcal{O}_L)$. Then, the following statements hold~:
\begin{itemize}
\item[i.] $a_0=0$ if and only if $v_p(\alpha_i)>0$ for all $i \geq 0$~;
\item[ii.] if $a_0\not=0$, then $v_p(a_0)=rv_p(\pi)$ if and only if $v_p(\alpha_i)>0$ for all $0\leq i <r$ and $v_p(\alpha_r)=0$.
\end{itemize}
\end{proposition}

\begin{proof}
We follow Pulita's proof of Lemma~2.2 in~(\cite{Pulita}, Sect.~2.1).  The two assertions can be recovered through the equivalence of the two following statements
\begin{itemize}
\item[(a)] $v_p(a_0)=rv_p(\pi)$~,
\item[(b)] $v_(\alpha_i)>0$ for all $0 \leq i < r$, and $v_p(\alpha_r)=0$~,
\end{itemize}
including the case $r=+\infty$, corresponding to $a_0=0$.

Given $k \geq 0$, condition~$2$ is equivalent to $v_p(\alpha_r^k)=0$ and $v_p(\alpha_i^k)>0$ for all $i<r$, which is again equivalent to condition~2 applied to the Witt vector $F^k(S_{P,a}(h))$ according to the assertion~iii of Proposition~\ref{WWF}. 

We write $(\beta_0,\beta_1,...)$ for the components of the Witt vector $F^k(S_{P,a}(h))$. Its ghost vector is $\langle h(P^{(k)}(a)),h(P^{(k+1)}(a)),...\rangle$, where $P^{(i)}$ denotes the polynomial $P$ composed $i$~times. 

Since $v_p(P(a))\geq \inf(qv_p(a),v_p(\pi)+v_p(a)) \geq v_p(a) >0$, we have that $v_p(P^{(k)}(a))\rightarrow +\infty$ as $k \rightarrow + \infty$. In particular, if $k$ is big enough, then $v_p(h(P^{(i)}(a)))=v_p(a_0)$ for all $i \geq k$.
Therefore, for such value of $k$, the relations between the components of $F^k(S_{P,a}(h))$ and its ghost components give us~:
$$\pi^j\beta_j=h(P^{(j+k)}(a))-(\beta_0^{q^j}+\pi \beta_1^{q^{j-1}}+...+\pi^{j-1}\beta_{j-1}^{q^2}) \quad \forall j \geq 0\enspace. $$
We thus see, by iteration on $j \geq 0$, that $v_p(a_0)=rv_p(\pi)$ if and only if $v_p(\beta_j)>0$ for all $j<r$ and $v_p(\beta_r)=0$, thereby proving the assertion. \end{proof}

\subsection{Formal group exponentials}

Again, $K$ is a finite extension of $\Q_p$, with valuation ring~$\bo_K$ and residue field $k$. We write $\text{card}(k)=q$ for some power $q$ of $p$. We fix a uniformising parameter $\pi$ of $\bo_K$. Let $P\in\mathcal{F}_{\pi}$ be some Lubin-Tate series with respect to $\pi$, and denote by $F_P\in \bo_K[[X,Y]]$ the unique formal group which admits $P$ as an endomorphism (see Subsection~\ref{fgroups}). By the functional equation lemma~(assertion~3 of Proposition~1.6), there exists a formal power series $g \in X\bo_K[[X]]$ with the coefficient of $X$ equal to $1$, and such that $F_P(X,Y)=f_g^{-1}(f_g(X)+f_g(Y))$, where $f_g \in XK[[X]]$ is given by Hazewinkel's functional equation construction applied to $g$, with uniformising parameter $\pi$ and residue cardinality $q$. Moreover, its composition inverse $f_g^{-1}\in XK[[X]]$ equals the exponential $\exp_{F_P}$ of the formal group $F_P$~: 
$$f_g^{-1}=\exp_{F_P}\enspace.$$

\

Now let $h(X)=X\in X\bo_K[[X]]$. Hazewinkel's functional equation construction applied to $h$, $\pi$ and $q$ then gives us $\displaystyle{f_h(X)=X+\frac{X^q}{\pi}+\frac{X^{q^2}}{\pi^2}+\cdots}$, and according to assertion~2 of Theorem~1.5, we have $f_g^{-1}(f_h(X))\in X \bo_K[[X]]$. Therefore, we can make the following definition~:

\begin{definition}\label{EP} Let $P \in \bo_K[[X]]$ be a Lubin-Tate series with respect to $\pi$. We define 
$$E_P(X):=\exp_{F_P}\left(X+\frac{X^q}{\pi}+\frac{X^{q^2}}{\pi^2}+\cdots \right) \in X\bo_K[[X]] \enspace.$$
\end{definition}

\

\begin{notation} Let $F$ be a formal group which endows a set $A$ with a group structure under the action $+_F$. We use the following sigma notation to denote the composition of multiple elements of $A$ using the group law $+_F$:
$$\sum_{j_0\le j\le j_n}^{F}a_j=a_{j_0}+_Fa_{j_0+1}+_F\cdots+_Fa_{j_n}\enspace,$$
where $j_0\in\Z_{\ge0}$ and $j_n\in\Z_{\ge0}\cup\{\infty\}$. Analogously to usual sums, the limits of infinite formal group sums might not always exist, and when they do, they might not be contained in $A$.
\end{notation}



\

Let $L$ be a finite extension of $K$. We provide the group $X\bo_L[[X]]$ with the ``$X$-adic'' topology induced by that of $\bo_L[[X]]$. By what precedes, if $\lambda=(\lambda_0,\lambda_1,...) \in W_{\bo_K, {\pi}}(\bo_L)$, the sum with respect to the formal group law $\displaystyle\sum^{F_P}_{0\le j\le\infty}E_P(\lambda_jX^{q^j})$ defines a formal power series in $X\bo_L[[X]]$. Generalising the Artin-Hasse exponential relative to ramified Witt vectors by the use of formal groups, we thus define~:

\

\begin{definition} \label{EPlambda}For every $\lambda \in W_{\bo_K,{\pi}}(\bo_L)$, the generalised Artin-Hasse exponential relative to $\lambda$ and $P$ is
$$E_P(\lambda,X):=\sum_{0\le j\le \infty}^{F_P}E_P(\lambda_jX^{q^j})=\exp_{F_P}\left(\lambda^{(0)}X+\lambda^{(1)}\frac{X^q}{\pi}+\lambda^{(2)}\frac{X^{q^2}}{\pi^2}+...\right) \in { X\bo_L[[X]]} \enspace.$$

\end{definition}

\

{We then fix} another uniformising parameter for $\bo_K$, {denoted by} $\pi'$, and let $Q\in \bo_K[[X]]$ be a Lubin-Tate series with respect to $\pi'$. We fix a coherent set of roots $\{\omega_i\}_{i>0}$
associated to $Q(X)$, \textit{i.e.}, a sequence of elements of $\bar\Q_p$ such
that $\omega_1\ne0$, $Q(\omega_1)=0$ and $Q(\omega_{i+1})=\omega_{i}$.

\

We also fix $n \geq 1$ and let $L=K_{\pi',n}$ be the $n$th Lubin-Tate extension of $K$ with respect to~$\pi'$. We note that $\bo_L=\bo_K[\omega_n]$. Let $h(X)=X$. {Since $\pi$ and $\pi'$ generate the same ideal in $\bo_K$, the polynomial~$Q$ satisfies Identities~\ref{PSpecificWittVectors}. In particular,} according to Lemma~\ref{WittKey} and Section~\ref{SpecificWittVectors}, the Witt vector {$S_{Q,\omega_n}(h)$} is well defined in $W_{\bo_K,\pi}(\bo_L)$ for every $n\geq 1$, and it is the unique Witt vector in ${W_{\bo_K,\pi}(\bo_L)}$ with ghost vector $\langle \omega_n,\omega_{n-1},...,\omega_1,0,...\rangle \in \bo_L^{\Z_{\geq 0}}$.

\

\begin{definition}\label{DefEPnQ}
We define $$E_{P,n}^Q(X):=E_P({S_{Q,\omega_n}(h)},X) \in X\bo_L[[X]] \enspace,$$
for $h(X)=X \in X\bo_L[[X]]$.
\end{definition}

As an interesting consequence of the properties of these power series, we can give an improvement to the known bound for the radius of convergence of the formal group exponential $\exp_{F_P}(X)$.

\begin{proposition}\label{ExpFPConv}
The power series $\exp_{F_P}(X)$ converges on the disc $\{x\in\mathbb{C}_p:v_p(x)>1/e_K(q-1)\}$, where $e_K=v_{\pi}(p)$ denotes the absolute ramification index of $K$ if $v_{\pi}$ is the discrete valuation on $K$ such that $v_{\pi}(\pi)=1$.
\end{proposition}

\begin{proof} First, with $h(X)=X$, we know $\exp_{F_P}(\omega_1X)
=E_P({S_{Q,\omega_1}(h)},X)=E_{P,1}^Q(X)$ is a formal power series with integral coefficients, so it converges at any element $x\in\mathbb C_p$ with strictly positive valuation. We know that $v_p(\omega_1)=1/(q-1)e_K$ as $\omega_1$ is a uniformising parameter for $K_{\pi',1}$, therefore $\exp_{F_P}(x)$ converges whenever $v_p(x)>1/(q-1)e_K$.\end{proof}

\begin{remark} The only bound of the radius of convergence of $\exp_{F_P}(X)$ known to the authors was that given by Fr\"ohlich in \cite[Ch.IV, Thm~3]{Frohlich}, which he accredits to Serre. This bound was $1/(p-1)$, so our bound improves this result for all $K\ne\Q_p$. For $K=\Q_p$ and $P(X)=(X+1)^p-1$, we obtain $\exp_{F_P}(X)=\exp(X)-1$ and we see that this bound is optimal. We conjecture that this bound is in fact optimal for all choices of $K$ and $P$.
\end{remark}

{One crucial argument in the proof of Theorem~\ref{MainTheorem} will be provided by the following lemma~:}

\begin{lemma}\label{DecE_P} Let $h(X)=X$. For every $\lambda=(\lambda_0,\lambda_1,...) \in {W_{\bo_K,\pi}(\bo_L)}$, the following equality holds~;
$$E_P({S_{Q,\omega_n}(h)}\lambda,X)=\sum_{0\le j\le n-1} ^{F_P}E_{P,n-j}^Q(\lambda_jX^{q^j}) \enspace.$$
\end{lemma}

\begin{proof}
{On the one hand, using Definition~\ref{EPlambda} and the multiplicativity of the ghost map for ramified Witt vectors, we get successively~:}
$${ \begin{array}{lcl}
E_P(S_{Q,\omega_n}(h)\lambda,X) 
 & = & \exp_{F_P}\left(\omega_n\lambda^{(0)}X+\omega_{n-1}\lambda^{(1)}\frac{X^q}{\pi}+\cdots + \omega_1\lambda^{(n-1)}\frac{X^{q^{n-1}}}{\pi^{n-1}} \right) \\
 & = &  \exp_{F_P}\left(\sum_{0\leq l \leq n-1}\omega_{n-l}\sum_{0\leq j \leq l}\lambda_j^{q^{l-j}}\frac{X^{q^l}}{\pi^{l-j}}\right) \\
& = &   \exp_{F_P}\left(\sum_{0\leq j \leq n-1}\sum_{0 \leq k \leq n-1-j}\omega_{n-j-k}\lambda_j^{q^k}\frac{X^{q^{j+k}}}{\pi^k}\right)
  \end{array}}$$

{On the other hand, using Definition~\ref{DefEPnQ}, Definition~\ref{EPlambda} and Identities~\ref{logexp_identitites}, we also get~:}
$$
{\begin{array}{lcl}
\displaystyle\sum_{0\le j\le n-1} ^{F_P}E_{P,n-j}^Q(\lambda_jX^{q^j})  & = & \displaystyle\sum_{0\le j\le n-1} ^{F_P}E_{P}(S_{Q,\omega_{n-j}}(h),\lambda_jX^{q^j}) \\
& = &  \displaystyle\sum_{0\le j\le n-1} ^{F_P} \exp_{F_P}(\omega_{n-j}\lambda_jX^{q^j}+\omega_{n-j-1}\lambda_j^q\frac{X^{q^{j+1}}}{\pi}+\cdots + \omega_1\lambda_j^{q^{n-j-1}}\frac{X^{q^{n-1}}}{\pi^{n-j-1}}) \\
&=& \exp_{F_P}(\sum_{0\le j\le n-1}\sum_{0\leq k \leq n-j-1}\omega_{n-j-k}\lambda_j^{q^k}\frac{X^{q^{j+k}}}{\pi^k}),
\end{array}}$$

{thereby proving the desired equality.}
\end{proof}

\begin{proposition}\label{EPOverConv} Let $h(X)=X \in \bo_K[[X]]$. For every Witt vector $\lambda {=(\lambda_0,\lambda_1,...)} \in {W_{\bo_K,\pi}(\bo_L)}$, if $v_p(\lambda_i)>0$ for all $i\in\{0,...,n\}$, the series $E_P({S_{Q,\omega_n}(h)}\lambda,X)$ is over-convergent
, \textit{i.e.}, it converges on the closed disk $\mathbb D=\{x \in \mathbb C_p, \: v_p(x) \geq 0\}$. Moreover, {$E_P(S_{Q,\omega_n}(h)\lambda,x)$} has strictly positive valuation for all $x \in \mathbb D$.
\end{proposition}

\begin{proof} According to Lemma~\ref{DecE_P}, the series $E_P({S_{Q,\omega_n}(h)}\lambda,X)$ is a finite sum with respect to the formal group law $F_P$. Therefore, it is over-convergent if and only if each term of the sum is over-convergent and has strictly positive valuation when evaluated at $x\in \mathbb C_p$ with $v_p(x)\geq 0$. But this is a consequence of the property that $E_{P,n-j}^Q(\lambda_jX^{q^j})\in \lambda_jX^{q^j}\bo_K[[X]]$. The last assertion is therefore trivial.
\end{proof}

\subsection{Proof of Theorem~\ref{MainTheorem}}
We can now prove our main theorem on the properties of the formal power series $\mathcal E_{P,n}^{Q}(X)$~:


\begin{proof1}

\noindent\textbf{Part 1.}
According to identity~\ref{logexp_identitites} of Subsection~\ref{fgroups}, we have
\begin{eqnarray}\nonumber
\mathcal{E}^Q_{P,n}(X)&=&\exp_{{F_P}}\left(\sum_{i=0}^{n-1}\frac{\omega_{n-i}(X^{q^i}-X^{q^{i+1}})}{\pi^{i}}\right)\\
&=&\exp_{F_P}(\pi\omega_{n+1}X)+_{F_P}\exp_{F_P}\left(\sum_{i=0}^{n}\omega_{n+1-i}\left(\frac{\omega_{n-i}}{\omega_{n+1-i}}-\pi\right)\frac{X^{q^i}}{\pi^i}\right)\nonumber\\
&=&\exp_{F_P}(\pi\omega_{n+1}X)+_FE_P({S_{Q,\omega_{n+1}}(h)}\lambda,X)\label{EQP}
\end{eqnarray}
with {$h(X)=X$ and } $\lambda=S_{{Q,\omega_{n+1}}}(f)$ for $f(X)=\frac{Q(X)}{X}-\pi$.

\

We have to prove that each term in this sum is over-convergent and has strictly positive valuation when evaluated at some $x\in\mathbb{C}_p$ with $v_p(x)\ge 0$. 

\

{We write $\lambda=(\lambda_0,\lambda_1,...)$}. First, the constant term of $f(X)$ is $\pi'-\pi$. Therefore, if $\pi\equiv\pi' \mod \mathcal P_K^{n+1}$, then $v_p(\lambda_i)>0$ for all $i\in\{0,...,n\}$ by Proposition~\ref{KeyProp}. So,  according to Proposition~\ref{EPOverConv},  the series {$E_P(S_{Q,\omega_{n+1}}(h)\lambda,X)$}, {which lies in $X\bo_K[\omega_{n+1}][[X]]$}, is over-convergent and has strictly positive valuation when evaluated at any $x$ with $v_p(x)\geq 0$.

\

On the other hand, 
we know from the proof of Proposition~\ref{ExpFPConv} {and Definition~\ref{DefEPnQ}} that $\text{exp}_{F_P}(\omega_1X)\in X \bo_K [\omega_1][[X]]$. Substituting $X$ with $(\pi\omega_{n+1}/\omega_1)X$ and noting that $v_p(\pi)>v_p(\omega_1)$, we then see that the series $\text{exp}_{F_P}(\pi\omega_{n+1}X)$ belongs to $(\pi\omega_{n+1}/\omega_1)X\bo_K{[\omega_{n+1}]}[[X]]$. In particular, it is over-convergent by Proposition~\ref{ExpFPConv}, and it has positive valuation when evaluated at any $x$ with $v_p(x)\geq 0$. 

\

{Therefore, the series $\mathcal{E}^Q_{P,n}(X)$ lies in $X\bo_K[\omega_{n+1}][[X]]$, and thus in $X\bo_K[\omega_n][[X]]$, since it belongs to $K_{\pi,n}[[X]]$ by definition. Moreover, it is over-convergent if $\pi\equiv \pi' \mod \mathcal P_K^{n+1}$, thereby proving Part~1. }

\

\noindent\textbf{Part 2.} 

In Part 1, we saw that $\exp_{F_P}(\pi{\omega_{n+1}}X)\in {(\pi\omega_{n+1}/\omega_1)X\bo_K[\omega_{n+1}][[X]]}$. We know that $v_p(\pi\omega_{n+1}/\omega_1)>v_p(\omega_n)$, therefore using expression (\ref{EQP}) above and the fact that $F_P\in\bo_K[[X,Y]]$ we are left with showing ${E_P(S_{Q,\omega_{n+1}}(h)\
 \lambda,X)}\equiv \omega_nX\mod\omega_n^2X\bo_K{[\omega_{n+1}]}[[X]]$.

\

Let $\bo_L$ be the valuation ring of some finite extension $L$ of $K$. Let $\nu=(\nu_i)_i \in W_{\bo_K,\pi}(\bo_L)$ and suppose that, for some $j$, $v_p(\nu_i)\ge v_p(\nu_j)$ for all $i\in \Z$. In Definition \ref{EP} we saw that $E_{P}(X)\in X\bo_K[[X]]$. Therefore, 
from Definition \ref{EPlambda} and the fact that $F_P(X,Y)$ has integral coefficients, we see that $E_P(\nu,X)\in\nu_j X\bo_L[[X]]$. 

\

{In particular, let $L=K(\omega_m)=K_{\pi',m}$ for some $m$ with $0<m\leq n$. For }$\nu=S_{{Q},\omega_m}(X)$, we know from Proposition \ref{KeyProp} that $v_p(\nu_i)>0$ for all $i$.  As $\nu\in W_{\bo_K {, \pi}}(\bo_K[\omega_m])$ we must then have $v_p(\nu_i)\ge v_p(\omega_m)$ for all $i$, and so 
$$E^{Q}_{P,j}(X)=E_P(S_{P,\omega_j}(X),X)\in\omega_jX\bo_{K}[\omega_j][[X]]\enspace.$$

\

Recall from Part 1 that $\lambda =(\lambda_i)_i=S_{{Q},\omega_{n+1}}(f)$ with $f(X)=\frac{Q(X)}{X}-\pi$ and that $v_p(\lambda_i)>0$ for all $i\in\{0,...,n\}$. Therefore, for $i\in\{0,...,n\}$ we have $v_p(\lambda_i)\ge v_p(\omega_{n+1})$, as $\lambda\in W_{\bo_K{, \pi'}}(\bo_K[\omega_{n+1}])$. We then have 
$$E^{Q}_{P,n+1-j}(\lambda_jX^{q^j})\in \lambda_j\omega_{n+1-j}X\bo_K[\omega_{n+1-j}][[X]]\enspace.$$
For all $1\le j\le n$ we have $v_p(\lambda_j\omega_{n+1-j})>v_p(\omega_n)$ and from Lemma \ref{DecE_P} we know that
$$E_P(S_{P,\omega_{n+1}}(X)\lambda,X)=\sum_{0\le j\le n} ^{F_P}E_{P,n+1-j}^Q(\lambda_jX^{q^j}) \enspace.$$

It is therefore sufficient to prove that $E_{P,n+1}^Q(\lambda_0X)\equiv \omega_nX\mod\omega_n\omega_{n+1}\bo_K[\omega_{n+1}][[X]]$.

By definition the coefficient of $X$ in the formal group exponential $\exp_{F_P}(X)$ is equal to~$1$. Therefore, if 
$$E^Q_{P,n+1}(X)=\exp_{F_P}\left(\omega_{n+1}X+\omega_{n}\frac{X^q}{\pi}+\cdots+\omega_1\frac{X^{q^{n}}}{\pi^n}\right)=\sum_{i\ge 1}a_iX^i\enspace,$$
then $a_1=\omega_{n+1}$. Recall that $a_i\in\omega_{n+1}\bo_K[\omega_{n+1}][[X]]$, thus
\begin{eqnarray*}E^Q_{P,n+1}(\lambda_0X)&\equiv&a_1\lambda_0X\mod\omega_{n+1}\lambda_0^2\bo_K[\omega_{n+1}[[X]]\\
&\equiv&\omega_{n+1}\lambda_0X\mod\omega_{n+1}\lambda_0^2\bo_K[\omega_{n+1}][[X]]\enspace.
\end{eqnarray*}
By definition, we see that $\lambda_0=\lambda^{(0)}=f(\omega_{n+1})=\omega_n/\omega_{n+1}-\pi$, and so 
$$E^Q_{P,n+1}(X)\equiv\omega_{n}X\mod(\omega_n^2/\omega_{n+1})\bo_K[\omega_{n+1}][[X]]\enspace,$$
which proves the result  since $v_p(\omega_n)>v_p(\omega_{n+1})$ {and since $\mathcal E_{P,n}^Q\in X\bo_K[\omega_n][[X]]$}.
\end{proof1}

\section{Applications}\label{applications}

We recall that $p$ is a rational prime, $K$ is a finite extension of $\Q_p$, $\pi$ and $\pi'$ are uniformising parameters of $\bo_K$ and $q$ is the cardinality of the residue field of $K$. 
We let $P\in\mathcal{F}_{\pi}$ (resp. $Q\in\mathcal{F}_{\pi'}$) be Lubin-Tate series with respect to $\pi$ (resp. $\pi'$), and $F_P(X,Y)$ (resp. $F_{Q}(X,Y)$) be the unique formal group with coefficients in $\bo_K$ that admits $P$ (resp. $Q$) as an endomorphism.

\subsection{Lubin-Tate division points}
Since the extensions $K_{\pi,n}$ are finite and abelian over $K$ for all choices of $n$ and $\pi$, we have $K_{\pi',n}=K_{\pi,n}$ exactly when their norm groups are equal \cite[Appendix, Theorem 9]{Washington}.
An exact description of the norm group $N_{K_{\pi,n}/K}(K_{\pi,n}^\times)$ was computed in \cite[Proposition 5.16]{Iwasawa}. 
Namely, $N_{K_{\pi,n}/K}(K_{\pi,n}^\times)=\langle\pi\rangle\times 1+\bp_K^{n}$. This implies $K_{\pi,n}=K_{\pi',n}$ if and 
only if $\pi\equiv\pi'\mod\bp_K^{n+1}$.

\

We now prove Proposition~\ref{LT_div_pt_thm}, which shows how values of the power series $\mathcal{E}^{Q}_{P,n}(X)$ give expressions for any $n$th Lubin-Tate division point with respect to $P$ in terms of $n$th Lubin-Tate division points with respect to $Q$ whenever $\pi\equiv\pi'\mod\bp_K^{n+1}$.


\begin{proof2}
We assume that $\pi\equiv\pi'\mod\bp_K^{n+1}$. We proceed by induction on $m$, with $0<m\leq n$.

From Identities \ref{logexp_identitites} and \ref{a_action} of Subsection~\ref{fgroups}, we have
\begin{eqnarray*}
[\pi]_P\left(\mathcal{E}^{Q}_{P,1}(X)\right)
&=&[\pi]_P\left(\exp_{F_P}\left(\omega_{1}(X-X^{q})\right)\right)\\
&=&\exp_{F_P}\left(\pi\omega_1(X-X^q)\right)\\
&=&\exp_{F_P}(\pi\omega_1X)-_{F_P}\exp_{F_P}(\pi\omega_1X^q)\enspace.
\end{eqnarray*}

From Proposition~\ref{ExpFPConv} we know that $\exp_{F_P}(X)$ converges on the disc $\{x\in\mathbb{C}_p~:v_p(x)>1/e_K(q-1)\}$. Also, in the proof of Proposition~\ref{ExpFPConv} {and according to Definition~\ref{DefEPnQ}}, we saw that $\exp_{F_P}(\omega_1 X)\in X\bo_K [\omega_1][[X]]$, and so $\exp_{F_P}(\pi\omega_1)$ will have positive valuation. 
We can therefore evaluate both the left and right hand side above at $1$ and get~:

$$[\pi]_P\left(\mathcal{E}^{Q}_{P,1}(1)\right)=\exp_{F_P}(\pi\omega_1)-_{F_P}\exp_{F_P}(\pi\omega_1)=0\enspace.$$
Therefore, $\mathcal{E}^{Q}_{P,1}(1)$ is a $[\pi]_P$-division point. 
Moreover, from Theorem~\ref{MainTheorem} Part 2 we know that $\mathcal{E}^{Q}_{P,1}(X)\equiv \omega_{1}X\mod \omega_1^2X\bo_{K}[\omega_1][[X]]$, and so this division point is primitive.

\

Now let $k\le m$ and assume the result holds for $m=1,\ldots k-1$. Again, using Identities \ref{logexp_identitites} and \ref{a_action}, we have
\begin{eqnarray*}
[\pi]_P\left(\mathcal{E}^{Q}_{P,m}(X)\right)
&=&[\pi]_P\left(\exp_{{F_P}}\left(\sum_{i=0}^{m-1}\frac{\omega_{m-i}(X^{q^i}-X^{q^{i+1}})}{\pi^i}\right)\right)\\
&=&\exp_{{F_P}}\left(\sum_{i=0}^{m-1}\frac{\pi\omega_{m-i}(X^{q^i}-X^{q^{i+1}})}{\pi^{i}}\right)\\
&=&\exp_{F_P}\left(\pi\omega_m(X-X^q)+\sum_{i=1}^{m-1}\frac{\omega_{m-i}(X^{q^{i}}-X^{q^{i+1}})}{\pi^{i-1}}\right)\\
&=&\exp_{F_P}\left(\pi\omega_m(X-X^q)+\sum_{i=0}^{(m-1)-1}\frac{\omega_{(m-1)-i}(X^{q^{i+1}}-X^{q^{i+2}})}{\pi^{i}}\right)\\
&=&\exp_{F_P}(\pi\omega_mX)-_{F_P}\exp_{F_P}(\pi\omega_mX^q)+_{F_P}\mathcal{E}^{Q}_{P,m-1}(X^q)\enspace.
\end{eqnarray*}

From Theorem~\ref{MainTheorem} Part 1 we know that if $\pi\equiv\pi'\mod\bp_K^{n+1}$, then $\mathcal{E}^{Q}_{P,m-1}(X^q)$ is over-convergent. From Proposition \ref{ExpFPConv} we know that $\exp_{F_P}(X)$ converges on the disc $\{x\in\mathbb{C}_p:v_p(x)>1/e_K(q-1)\}$. Therefore, all the power series on the right hand side of this equation are overconvergent. Also, $\exp_{F_P}(\pi\omega_m)$ and $\mathcal{E}^{Q}_{P,m-1}(1)$ both have positive valuations, so the formal group operations on these values are well defined. We can therefore evaluate both the left and right hand side at $1$~:

\begin{equation}
[\pi]_P\left(\mathcal{E}^{Q}_{P,m}(1)\right)
=\exp_{F_P}(\pi\omega_m)-_{F_P}\exp_{F_P}(\pi\omega_m)+_{F_P}\mathcal{E}^{Q}_{P,m-1}(1)
=\mathcal{E}^{Q}_{P,m-1}(1)\enspace.
\label{coherent}
\end{equation}

By the induction hypothesis, we know that $\mathcal{E}^{Q}_{P,m-1}(1)$ is a primitive $[\pi^{m-1}]_P$-division point and therefore $\mathcal{E}^{Q}_{P,m}(1)$ is a primitive $[\pi^m]_P$-division point. Part 1 now follows by induction.

Part 2 then follows directly from Part 1 and Equation~\ref{coherent} above.
\end{proof2}


\subsection{Galois action on $\mathcal{E}^{Q}_{P,n}(1)$}
From Proposition~\ref{LT_div_pt_thm} we know that $K_{\pi,n}=K(\mathcal{E}_{P,n}^{Q}(1))$.
We will now give a complete description of how $\Gal(K_{\pi,n}/K)$ acts on $\mathcal{E}_{P,n}^{Q}(1)$.

\

We know that $\mathcal{E}_{P,n}^{Q}(1)$ is a primitive $n$th Lubin-Tate division point with respect to $P$. From standard theory (see \cite[\S6-7]{Iwasawa}, specifically Theorem 7.1) we know that the elements of $\Gal(K_{\pi,n}/K)$ are those automorphisms such that 
$\mathcal{E}_{P,n}^{Q}(1)\mapsto [u]_P(\mathcal{E}_{P,n}^{Q}(1))$,
where $u$ runs over a set of representatives of $\bo_K^{\times}/(1+\bp_K^n)$, for example 
$$\left\{\sum_{i=0}^{n-1}z_i\pi^i:z_i\in\mu_{q-1}\cup\{0\}, z_0\ne0\right\}\enspace.$$
We now prove Proposition~\ref{thm3}, which gives us a complete description of $\Gal(K_{\pi,n}/K)$ in terms of values of our power series.

 \begin{proof3}
 From the definition of $\mathcal E^{Q}_{P,n}(X)$ and Identity \ref{a_action} of Subsection~\ref{fgroups}, we have
\begin{eqnarray}
[\sum_{i=0}^{n-1}z_i\pi^i]_P(\mathcal E^{Q}_{P,n}(X)) 
&=&[\sum_{i=0}^{n-1}z_i\pi^i]_P\left(\exp_{{F_P}}\left(\sum_{i=0}^{n-1}\frac{\omega_{n-i}(X^{q^i}-X^{q^{i+1}})}{\pi^i}\right)\right)\nonumber \\
&=&\exp_{{F_P}}\left(\left(\sum_{j=0}^{n-1}z_j\pi^j\right)\sum_{i=0}^{n-1}\frac{\omega_{n-i}(X^{q^i}-X^{q^{i+1}})}{\pi^i}\right)\enspace.\nonumber\end{eqnarray}

Using Identity \ref{logexp_identitites} and the observation that $z_j=z_j^{q^i}$ for all $i$ and $j$, we then see that this is equal to
\begin{equation}\sum_{0\le j\le n-1}^{F_P}\exp_{{F_P}}\left(\sum_{i=0}^{n-1}\frac{\omega_{n-i}((z_{j}X)^{q^i}-(z_{j}X)^{q^{i+1}})}{\pi^{i-j}}\right)\enspace.\label{second_exp}\end{equation}

We now develop each term of this expression. For all $1\le j\le n-1$, we get~:
\begin{eqnarray}
&& \exp_{{F_P}}\left(\sum_{i=0}^{n-1}\frac{\omega_{n-i}((z_{j}X)^{q^i}-(z_{j}X)^{q^{i+1}})}{\pi^{i-j}}\right) \\
&=& \exp_{F_P}\left(\sum_{i=0}^{j-1}\frac{\omega_{n-i}((z_jX)^{q^i}-(z_jX)^{q^{i+1}})}{\pi^{i-j}}+\sum_{i=j}^{n-1}\frac{\omega_{n-i}((z_jX)^{q^i}-(z_jX)^{q^{i+1}})}{\pi^{i-j}}  \right)\nonumber\\
&=& \exp_{F_P}\left(\sum_{i=0}^{j-1}\pi^{j-i}\omega_{n-i}((z_jX)^{q^i}-(z_jX)^{q^{i+1}})\right)+_{F_P}\exp_{F_P}\left(\sum_{i=0}^{n-j-1}\frac{\omega_{n-j-i}((z_jX)^{q^{i+j}}-(z_jX)^{q^{i+j+1}})}{\pi^{i}}  \right)\nonumber\\
&=&\mathcal{E}^{Q}_{P,n-j}((z_jX)^{q^j})+_{F_P}\sum_{0\le i\le j-1}^{F_P}\left(\exp_{F_P}(\pi^{j-i}\omega_{n-i}z_jX^{q^i})-_{F_P}\exp_{F_P}(\pi^{j-i}\omega_{n-i}z_jX^{q^{i+1}})\right)\label{final_exp}\enspace.
\end{eqnarray}

Similarly to before, from Theorem~\ref{MainTheorem} Part 1 we know that if $\pi\equiv\pi'\mod\bp_K^{n+1}$, then $\mathcal{E}^{Q}_{P,n-j}(X)$ is over-convergent for all $0\le j\le n-1$ and therefore, $\mathcal{E}^{Q}_{P,n-j}((z_jX)^{q^j})$ is over-convergent for all $0\le j\le n-1$. From Proposition \ref{ExpFPConv} we know that $\exp_{F_P}(X)$ converges on the disc $\{x\in\mathbb{C}_p:v_p(x)>1/e_K(q-1)\}$. Therefore, all the power series in (\ref{final_exp}) are over-convergent. Also all the power series in (\ref{final_exp}) have positive valuations when evaluated at $1$, so we can use formal group operations on these values. Evaluating (\ref{final_exp}) at $1$ we get $\mathcal{E}^{Q}_{P,n-j}(z_j)$. Combining this with (\ref{second_exp}) we then get
$$[\sum_{i=0}^{n-1}z_i\pi^i]_P(\mathcal E^{Q}_{P,n}(1))=\mathcal E^{Q}_{P,n}(z_0)+_{F_P}\mathcal E^{Q}_{P,n-1}(z_1)+_{F_P}\ldots+_{F_P}\mathcal E^{Q}_{P,1}(z_{n-1})\enspace.$$\end{proof3}
 
 \subsection{Local Galois module structure in weakly ramified extensions}\label{Galmod}

As mentioned in the introduction, one of the main motivations for the generalisation of Dwork and Pulita's power series has come from recent progress with open questions on Galois module structure. 

\

First, let $E/F$ be a finite odd degree Galois extension of number fields, with Galois group~$G$ and rings of integers~$\bo_E$ and $\bo_F$. From
Hilbert's formula for the valuation of the different $\D_{E/F}$ (\cite{serre}, IV, \S2, Prop.4), we know that the valuation of $\D_{E/F}$ will be even at every prime ideal of $\bo_E$ and we can define the
square-root of the inverse different $\A_{E/F}$ to be the unique
fractional $\bo_E$-ideal such that  
$$\A_{E/F}^2=\D_{E/F}^{-1}\enspace.$$
Erez has  proved that $\A_{E/F}$ is locally free over
$\bo_F[G]$ if and only if $E/F$ is at most weakly ramified, \textit{i.e.}, the second ramification groups are trivial at every prime \cite{erez2}~; however, the question of whether $\A_{E/F}$ is free over $\Z[G]$ still remains open. 
The tame case has been solved by Erez~\cite{erez2}. Now, it is possible for weakly ramified extensions to be wildly ramified, and here new obstructions arise.

\vskip1mm

Using Fr\"ohlich's classic \textit{Hom-description} approach, see \cite{Frohlich-Alg_numb}, it is possible to reduce this problem to carrying out calculations at a local level. The key to these local calculations is to have an explicit description of an integral normal basis generator for the square-root of the inverse different in weakly ramified extensions of local fields. It is also possible to reduce this problem further to considering only totally ramified extensions (see \cite[\S6]{erez2} and \cite[\S3]{PickettVinatier}). 

\vskip1mm

Precisely, let $M$ be the unique degree $p$ extension of $\Q_p$ contained in $\Q_p(\zeta_{p^2})$; this extension is totally, weakly and wildly ramified. First, in \cite{erez}, Erez proves that the element
$$\frac{1+Tr_{\Q_p(\zeta_{p^2})/M}(\zeta_{p^2})}{p}$$ 
is an integral normal basis generator for the square-root of the inverse different of $M/\Q_p$. In~\cite{Vinatier_jnumb}, Vinatier uses Erez's basis to prove that $\A_{E/F}$ is free over $\Z[G]$ with $F=\Q$ whenever the 
decomposition group at every wild place is abelian. Then, in~\cite{Pickett}, Pickett uses the trace map and special values of 
Dwork's power series to generalise Erez's basis to degree~$p$ extensions of an unramified extension of $\Q_p$ that are 
contained in certain Lubin-Tate extensions. In~\cite{PickettVinatier}, Pickett and Vinatier use Pickett's bases to prove that 
$\A_{E/F}$ is free over $\Z[G]$ under certain conditions on both the decomposition groups and base field. 

\vskip1mm

Following these results, we shall give explicit descriptions of integral normal basis generators for the square-root of the inverse different in abelian totally, weakly and wildly ramified extensions of any finite extension of $\Q_p$.
 
 \
 
 Another application is concerned with the Galois module structure of the valuation ring over its associated order in extensions of local fields. Precisely, let $L/K$ be a finite Galois extension of $p$-adic fields, with Galois group $G$. We denote by $\bo_K\subset \bo_L$ the corresponding valuation rings, and by $\mathfrak A_{L/K}$ the associated order of $\bo_L$ in the group algebra $K[G]$, that is 
$$\mathfrak A_{L/K}=\{\lambda \in K[G] \: : \: \lambda \bo_L \subset \bo_L\}.$$
This is an $\bo_K$-order of $K[G]$, and the unique one over which $\bo_L$ could be free as a module. When the extension $L/K$ is  at most tamely ramified, the equality $\mathfrak A_{L/K}=\bo_K[G]$ holds, and $\bo_L$ is $\mathfrak A_{L/K}$-free according to Noether's criterion. But when wild ramification is permitted, the structure of $\bo_L$ 
as an $\mathfrak A_{L/K}$-module is much more difficult to determine (see, e.g.,~\cite{Lara10} for an exposition of recent progress in this topic). 
A $p$-adic version of Leopoldt's theorem asserts that the ring $\bo_L$ is $\mathfrak A_{L/K}$-free whenever $K=\Q_p$ and $G$ is abelian. However, the field $\Q_p$ is actually the only base field which satisifes this property. One extension of this result is due to Byott~(\cite{Byott-Int_Gal_Mod_Struc_Some_LubinTate}, Cor.~4.3): If $L/K$ is an abelian extension of $p$-adic fields, then $\bo_L$ is free as a module over its associated order $\mathfrak A_{L/K}$ whenever the extension $L/K$ is totally and weakly ramified. We shall construct explicit generators of the valuation ring over its associated order in maximal abelian totally, weakly and wildly ramified extensions of $K$, using the description of such extensions that comes from Proposition~\ref{AbTotWildWeak} below.
 
 \

 These applications are closely related to each other. Their content is resumed in Theorem~\ref{Thm4}, which we prove in this section.
 
 \

Again, we denote by $K$ a finite extension of $\Q_p$, with valuation ring $\bo_K$, maximal ideal $\bp_K$, and residue field $k$.  We write $\text{card}(k)=q$. Let $\pi$ be a uniformising parameter of $K$. By standard theory we know that the Lubin-Tate extension $K_{\pi,2}/K$ is abelian and that $[K_{\pi,2}:K]=q(q-1)$. We can therefore define $M_{\pi,2}$ as the unique sub-extension of $K_{\pi,2}/K$ such that $[M_{\pi,2}:K]=q$. 
The following result is a direct consequence of Theorem~1.1 of \cite{PickettVinatier2}, however the present paper was not published at the time of writing, so we include another proof.
 
 \begin{proposition}\label{AbTotWildWeak}
 Every maximal abelian totally, wildly and weakly ramified extension of $K$ is equal to $M_{\pi,2}$ for some uniformising parameter $\pi$.
 \end{proposition}
 
 \begin{proof}
 Let $M$ be an abelian totally, wildly and weakly ramified extension of $K$.
  From \cite[Lemma 4.2]{Byott-Int_Gal_Mod_Struc_Some_LubinTate}, $M$ must be contained in $K_{\pi,2}$ for some $\pi$. Since it is of degree a power of $p$ over $K$, it is thus contained in $M_{\pi,2}$.
  
  Now, the extension  $M_{\pi,2}/K$ is clearly abelian totally and wildly ramified.  We are thus left with proving that it is weakly ramified. Since this extension is totally and wildly ramified, the numbers $-1$ and $0$ are neither lower jumps, nor upper jumps. By standard Lubin-Tate theory (see \cite{Iwasawa}, \S7), we know that the jumps of $K_{\pi,2}/K$ occur at 0 and 1 in the upper numbering. Therefore, by Herbrand's theorem, $1$ is the only upper jump (and so lower jump) of $M_{\pi,2}/K$, as required.
   \end{proof}

We can now prove Theorem~\ref{Thm4}.

 \begin{proof4}
 As $M_{\pi,2}/K$ is totally ramified and of degree $q$, we know that $e(M_{\pi,2}/K)=q$. Therefore, from \cite[Theorem~1]{Byott-Int_Gal_Mod_Struc_Some_LubinTate} we know that any element of $M_{\pi,2}$ with valuation $1$ must be a generator of $\mathcal \bo_{M_{\pi,2}}$ over its associated order in $M_{\pi,2}/K$, and from
 \cite[Corollary 2.5(i)]{Vinatier-3} that any element of $M_{\pi,2}$ with valuation $1-q$ must be an integral normal basis generator for $\A_{M_{\pi,2}/K}$. Therefore, it suffices to prove that the trace element $Tr_{K_{\pi,2}/M}(\mathcal{E}^{Q}_{P,2}(1))$ is a uniformising parameter of $M_{\pi,2}$.
 
 \

 First, the polynomial $f(X):=\frac{P(P(X))}{P(X)}$ is of degree $q(q-1)$, and it is irreducible since it is Eisenstein over~$K$: Indeed, $f(X)=P(X)^{q-1}+\sum_{i=2}^{q-1}a_iP(X)^{i-1}+\pi$ and each $a_i$ is divisible by $\pi$ because $P$ is a Lubin-Tate polynomial. Therefore, since $f(\mathcal{E}^Q_{P,2}(1))=0$, $f$ is the minimal polynomial of $\mathcal{E}_{P,2}^Q(1)$. Thus, $Tr_{K_{\pi,2}/K}(\mathcal{E}^{Q}_{P,2}(1))$ is the coefficient of $X^{q(q-1)-1}$ in $f(X)$, \textit{i.e.}, $Tr_{K_{\pi,2}/K}(\mathcal{E}^{Q}_{P,2}(1))=(q-1)a_{q-1}$. In particular, $Tr_{K_{\pi,2}/K}(\mathcal{E}^{Q}_{P,2}(1))$ has valuation~$1$ in $K$.

 \

 Now, we denote by $d$ the valuation of the different of the extension $M_{\pi,2}/K$. According to the characterisation of the different \cite[Ch. III, \S 3, Prop.~7]{serre}, we have 
 $$\forall i \in \Z, \quad Tr_{M_{\pi,2}/K}(\bp_{M_{\pi,2}}^{-2d+i})=\bp_K^{\lfloor \frac{i}{e_{M_{\pi,2}/K}} \rfloor},$$
 where $\lfloor x \rfloor$ denotes the largest integer $n$ with $n \leq x$. 
 
 \
 
 Since the extension $M_{\pi,2}/K$ is totally and weakly ramified of degree~$q$, we have $d=2(q-1)$, by Hilbert's formula for the valuation of the different. In particular, for $i=2q-1$, the previous relation gives the identity $Tr_{M_{\pi,2}/K}(\bp_{M_{\pi,2}})=\bp_K$, and for $i \geq 2q$, it proves that $Tr_{M_{\pi,2}/K}(\bp_{M_{\pi,2}}^2) \subseteq \bp_K^2$.
 
 \

 On the other hand, since the extension $K_{\pi,2}/M_{\pi,2}$ is tamely ramified, and because $\mathcal E_{P,2}^Q(1)$ is a uniformising parameter of $K_{\pi,2}$ by Theorem~2, we know that the element $Tr_{K_{\pi,2}/M_{\pi,2}}(E_{P,2}^Q(1))$ lies in $\bp_{M_{\pi,2}}$.
 
 \

 Now, from the transitivity of the trace we also have
 $$Tr_{K_{\pi,2}/K}(\mathcal{E}^{Q}_{P,2}(1))=Tr_{M_{\pi,2}/K}(Tr_{K_{\pi,2}/M_{\pi,2}}(\mathcal{E}^{Q}_{P,2}(1)))\enspace, $$ 
 and so, we can conclude that $Tr_{K_{\pi,2}/M_{\pi,2}}(\mathcal{E}^{Q}_{P,2}(1))$ is indeed a uniformising parameter in the field $M_{\pi,2}$. 
 \end{proof4}
 
 \

 We shall close this paper with the following corollary to Theorem~\ref{Thm4}, which gives explicit integral normal basis generators for the square root of the inverse different 
 in every abelian totally, wildly and weakly ramified extensions of~$K$. 
 
 \begin{corollary}\label{CorThm2}[Corollary to Theorem~\ref{Thm4}] Let $M$ be an intermediate subfield of the extension $M_{\pi,2}/K$. Let $\alpha_{\pi,2}$ equal either $\frac{Tr_{K_{\pi,2}/M}(\mathcal{E}^{Q}_{P,2}(1))}{\pi}$ or $\frac{Tr_{K_{\pi,2}/M}(\mathcal{E}^{Q}_{P,2}(1))+q}{\pi}$. 
 
 Then, 
 $\alpha_{\pi,2}$ is an integral normal basis generator for $\A_{M/K}$.
 \end{corollary}
 
 \begin{proof}
 This follows directly from Theorem~\ref{Thm4} and (\cite{Vinatier-3}, Corollary 2.5(ii)).
 \end{proof}
 
 \bibliography{bib}

\begin{thebibliography}{10}

\bibitem{Bourbaki83}
N.~Bourbaki.
\newblock {\em Alg\`ebre commutative}.
\newblock Masson, Paris, 1983.

\bibitem{Byott-Int_Gal_Mod_Struc_Some_LubinTate}
N.~P. Byott.
\newblock Integral {G}alois {M}odule {S}tructure of {S}ome {L}ubin-{T}ate
  {E}xtensions.
\newblock {\em J. Number Theory}, 77:252--273, 1999.

\bibitem{Ditters75}
E.~J. Ditters.
\newblock Formale gruppen, die vermutungen von atkin-swinnerton-dyer und
  verzweigte witt-vektoren.
\newblock {\em Lecture Notes, G\"ottingen}, 1975.

\bibitem{Drinfeld76}
V.~G. Drinfel'd.
\newblock Coverings of {$p$}-adic symmetric domains.
\newblock {\em Funkcional. Anal. i Prilo\v zen.}, 10(2):29--40, 1976.

\bibitem{erez}
B.~Erez.
\newblock The {G}alois {S}tructure of the {T}race {F}orm in {E}xtensions of
  {O}dd {P}rime {D}egree.
\newblock {\em J. Algebra}, 118:438--446, 1988.

\bibitem{erez2}
B.~Erez.
\newblock The {G}alois {S}tructure of the {S}quare {R}oot of the {I}nverse
  {D}ifferent.
\newblock {\em Math.Z.}, 208:239--255, 1991.

\bibitem{Frohlich}
A.~Fr{\"{o}}hlich.
\newblock {\em Formal Groups}.
\newblock Number~74 in Lecture Notes in Math. Springer Verlag, New York, 1968.

\bibitem{Frohlich-Alg_numb}
A.~Fr{\"{o}}hlich.
\newblock {\em Galois Module Structure of Algebraic Integers}.
\newblock Springer-Verlag, 1983.

\bibitem{Hazewinkel_book}
M.~Hazewinkel.
\newblock {\em Formal groups and applications}, volume~78 of {\em Pure and
  Applied Mathematics}.
\newblock Academic Press Inc. [Harcourt Brace Jovanovich Publishers], New York,
  1978.

\bibitem{HazewinkelWittVectors}
M.~Hazewinkel.
\newblock Twisted {L}ubin-{T}ate formal group laws, ramified {W}itt vectors and
  (ramified) {A}rtin-{H}asse exponentials.
\newblock {\em Trans. Amer. Math. Soc.}, 259(1):47--63, 1980.

\bibitem{Iwasawa}
K.~Iwasawa.
\newblock {\em Local Class Field Theory}.
\newblock Oxford University Press, 1986.

\bibitem{FontaineFargues}
J.-M.~Fontaine L.~Fargues.
\newblock Courbes et fibr\' es vectoriels en th\' eorie de hodge $p$-adique.
\newblock {\em preprint}, 2011.

\bibitem{Lang}
S.~Lang.
\newblock {\em Algebra}.
\newblock Addison-Wesley, 1965.

\bibitem{LangII}
S.~Lang.
\newblock {\em Cyclotomic Fields II}.
\newblock Springer-Verlag, New York, 1980.

\bibitem{Lubin_Tate}
J.~Lubin and J.~Tate.
\newblock Formal complex multiplication in local fields.
\newblock {\em Ann. of Math. (2)}, 81:380--387, 1965.

\bibitem{Pickett}
E.~J. Pickett.
\newblock Explicit {C}onstruction of {S}elf-{D}ual {I}ntegral {N}ormal {B}ases
  for the {S}quare-{R}oot of the {I}nverse {D}ifferent.
\newblock {\em J. Number Theory}, 129:1773 -- 1785, 2009.

\bibitem{PickettVinatier2}
E.~J. Pickett and S.~Vinatier.
\newblock Exponential {P}ower {S}eries, {G}alois {M}odule {S}tructure and
  {D}ifferential {M}odules.
\newblock Submitted.

\bibitem{PickettVinatier}
E.~J. Pickett and S.~Vinatier.
\newblock {S}elf-{D}ual {I}ntegral {N}ormal {B}ases and {G}alois {M}odule
  {S}tructure.
\newblock Submitted.

\bibitem{Pulita}
A.~Pulita.
\newblock Rank one solvable {$p$}-adic differential equations and finite
  abelian characters via {L}ubin-{T}ate groups.
\newblock {\em Math. Ann.}, 337(3):489--555, 2007.

\bibitem{serre-lubintate}
J.~P. Serre.
\newblock Local class field theory.
\newblock In J.~W.~S. Cassels and A.~Fr\"{o}hlich, editors, {\em Algebraic
  Number Theory}, London, 1967. Academic Press.

\bibitem{serre}
J.~P. Serre.
\newblock {\em Corps Locaux}.
\newblock Hermann, Paris, 1968.

\bibitem{Lara10}
L.~Thomas.
\newblock On the {G}alois module structure of extensions of local fields.
\newblock {\em Publ. Math. Besan\c con}, pages 157--194, 2010.
\newblock Actes de la {C}onf\'erence ``{F}onctions {$L$} et {A}rithm\'etique''.

\bibitem{Vinatier_jnumb}
S.~Vinatier.
\newblock Structure galoisienne dans les extensions faiblement ramifi\'ees de
  {$\Bbb Q$}.
\newblock {\em J. Number Theory}, 91(1):126--152, 2001.

\bibitem{Vinatier-3}
S.~Vinatier.
\newblock Galois {M}odule {S}tructure in {W}eakly {R}amified 3-{E}xtensions.
\newblock {\em Acta Arith.}, 119(2):171--186, 2005.

\bibitem{Washington}
L.~C. Washington.
\newblock {\em Introduction to Cyclotomic Fields}, volume~83 of {\em Graduate
  Texts in Mathematics}.
\newblock Springer-Verlag, New York, 1982.

\bibitem{Witt36}
E.~Witt.
\newblock Zyklische k\"orper und algebren der charakteristik vom grad $p^n$.
\newblock {\em J. Reine Angew. Math.}, 174:126--140, 1936.

\end{thebibliography}
\end{document}